\documentclass[twoside]{article}
\usepackage[a4paper]{geometry}
\usepackage[latin1,utf8]{inputenc} 
\usepackage[T1]{fontenc} 
\usepackage{RR}
\usepackage{amsmath,amssymb,amsfonts,amsthm}
\usepackage{bm} % for bold math symbols
\usepackage{mathtools}
\usepackage{algorithm} % for pretty algorithms
\mathtoolsset{showonlyrefs,centercolon} % only number referenced equations, better aligned colon in math mode
\usepackage[maxnames=99,sorting=none,giveninits=true,style=numeric-comp]{biblatex}
\usepackage{empheq} % sub-numbering of cases in equations
\usepackage{xr-hyper}
\usepackage{xr}
\usepackage[colorlinks,linktocpage]{hyperref}
\usepackage{graphicx}
\graphicspath{{./}}
\usepackage{subfig} % for multiple subfigures
\usepackage{enumitem} % enum with custom label
\usepackage{tensor} % for pre-pended subscript/superscript

\DeclareMathOperator{\vol}{vol}
\DeclareMathOperator{\sign}{sign}
\DeclareMathOperator{\aff}{aff}
\DeclareMathOperator{\ch}{conv}

\DeclareMathOperator*{\argmin}{arg\,min}

\newcommand{\df}{:=}
\newcommand{\fd}{=:}

\newcommand{\indicf}[1]{\mathbf{1}_{#1}}
\newcommand{\subst}[3]{\,\tensor*[^{#3}_{#2}]{{#1}}{}}
\newcommand{\scal}[2]{\left\langle{#1},{#2}\right\rangle}
\newcommand{\M}[2]{M({#1}\mid {#2})}
\newcommand{\detp}[1]{\det(#1)}
\newcommand{\dets}[3]{\det(\subst{#1}{#2}{#3})}
\newcommand{\blank}{{\,\cdot\,}}
\newcommand{\dnorm}[1]{\left\Vert#1\right\Vert}
\newcommand{\norm}[1]{\left|#1\right|}
\newcommand{\multichoose}[2]{
\left.\mathchoice
  {\left(\kern-0.48em\binom{#1}{#2}\kern-0.48em\right)}
  {\big(\kern-0.30em\binom{\smash{#1}}{\smash{#2}}\kern-0.30em\big)}
  {\left(\kern-0.30em\binom{\smash{#1}}{\smash{#2}}\kern-0.30em\right)}
  {\left(\kern-0.30em\binom{\smash{#1}}{\smash{#2}}\kern-0.30em\right)}
\right.}

\RRNo{9350}
\RRdate{June 2020}
\RRversion{3}
\RRdater{March 2021}
\RRauthor{% les auteurs
Hélène Barucq%
  \thanks[inria]{Inria Bordeaux Sud-Ouest, Project-Team Makutu, Inria - E2S UPPA-CNRS, Pau, France}%
\and Henri Calandra\thanks[total]{Total S.E., Centre Scientifique et Technique Jean Féger, Pau, France}%
\and Julien Diaz\thanksref{inria}
\and\\Stefano Frambati%\thanks[total]{Total SA, Centre Scientifique et Technique Jean Féger, avenue Larribau, Pau, France}
\thanksref{total}
\thanksref{inria}\thanks{Corresponding author}}
\authorhead{Barucq, Calandra, Diaz, Frambati}
\RRtitle{Espaces de splines réproduisant les polynômes à partir de pavages de zonotopes}
\RRetitle{Polynomial-reproducing spline spaces from fine zonotopal tilings}
\titlehead{Polynomial-reproducing spline spaces from fine zonotopal tilings}
\RRresume{Étant donné une configuration de points $A$, on explore une connexion entre les espaces de splines reproduisant les polynômes sur certains sous-ensembles de $\ch(A)$ et les pavages fins du zonotope $Z(V)$ associé à la configuration de vecteurs correspondante. Ce lien généralise directement un résultat connu sur les configurations de Delaunay et inclut naturellement, grâce à son charactère combinatoire, le cas de points en répétés et affinement dépendants en $A$. On prouve l'existence d'un processus de construction itératif général pour ces espaces. Enfin, on tourne notre attention vers les pavages de zonotopes fins et {\em réguliers}, en spécialisant nos résultats précédentes et en exploitant le graphe d'adjacence du pavage afin de proposer un ensemble d'algorithmes utiles en pratique pour la construction et l'évaluation des fonctions splines associées.}
\RRabstract{Given a point configuration $A$, we uncover a connection between polynomial-reproducing spline spaces over subsets of $\ch(A)$ and fine zonotopal tilings of the zonotope $Z(V)$ associated to the corresponding vector configuration. This link directly generalizes a known result on Delaunay configurations and naturally encompasses, due to its combinatorial character, the case of repeated and affinely dependent points in $A$. We prove the existence of a general iterative construction process for such spaces. Finally, we turn our attention to {\em regular} fine zonotopal tilings, specializing our previous results and exploiting the adjacency graph of the tiling to propose a set of practical algorithms for the construction and evaluation of the associated spline functions.}
\RRmotcle{spline multivariée, spline simplexe, base de splines, pavage de zonotope}
\RRkeyword{multivariate spline, simplex spline, spline basis, zonotopal tiling}
\RRprojet{Makutu}  % cas d'un seul projet
\RCBordeaux % Sophia Antipolis M\'editerran\'ee

\newtheorem{theorem}{Theorem}[section]
\newtheorem{lemma}[theorem]{Lemma}
\newtheorem{proposition}[theorem]{Proposition}
\newtheorem{corollary}[theorem]{Corollary}
\newtheorem{definition}[theorem]{Definition}
\newtheorem{remark}[theorem]{Remark}

\numberwithin{equation}{section}

\begin{document}

\makeRR   % cas d'un rapport de recherche
%% \makeRT % cas d'un rapport technique.
%% a partir d'ici, chacun fait comme il le souhaite

\setcounter{tocdepth}{2}
%    Text of article.
\section{Introduction}
\label{sec:intro}
Curves and surfaces based on piecewise-polynomial Bézier and B-spline functions \cite{piegl2012nurbs,prautzsch2013bezier} have long been invaluable tools in computer-aided design, computer graphics, machining and fabrication and, more recently, numerical analysis of partial differential equations \cite{hughes2005isogeometric}. The feature of reproducing all the polynomials over an interval up to a given degree underpins their use as approximation and interpolation tools. In one dimension, many robust and efficient evaluation schemes have become available to efficiently construct and evaluate these families of functions. In two or more dimensions, spline functions can be constructed via tensor products of one-dimensional B-splines, but this structure can be too rigid in some applications. For this reason, much work has gone into the direct generalization of B-spline functions to a multivariate setting. While natural generalizations of single B-spline functions have been found \cite{neamtu2001natural}, current state-of-the-art approaches for unstructured splines are still somewhat lacking: the main construction algorithm \cite{liu2007quadratic,liu2008computations} is only proven to work in two dimensions, and has only recently \cite{schmitt2019bivariate} been shown to converge for all degrees. Moreover, the current formulations fall short of treating the case of repeated and affinely dependent knots, which is needed in many practical applications. No simple and general evaluation scheme is known for multivariate spline spaces.

In this work, we set out to improve on some of these shortcomings by showing how these bases can be recast in a more general combinatorial form, paving the way for their use in efficient numerical schemes. We base our formulation on a connection between simplex spline spaces and fine zonotopal tilings, whose combinatorial nature allows a unified treatment free of the degenerate configurations that are typical of a purely geometrical approach. Furthermore, these structures come equipped with a natural adjacency graph, which can be used to navigate between splines in a basis and extend some aspects of the classical one-dimensional evaluation scheme to higher dimensions. This removes, in our view, one important computational shortcoming that prevented a more widespread use of these functions. 

Finally, note that some (unrelated) connections between zonotopal tilings and box splines have been drawn in the past, see e.g. \cite{de2010topics}. 

\subsection{Notation}
We adopt some standard notation from combinatorial geometry. Specifically, given $n\in\mathbb{Z}^+$, we define the range $[n]\df\{1,\ldots,n\}$. % and, for all $0\leq k\leq n$, we definef $\binom{[n]}{k}\df\{Q\subseteq [n]:\norm{Q}=k\}$.
The union between two disjoint sets $R$ and $S$ is denoted by $R\sqcup S$. Note that $\norm{R\sqcup S}=\norm{R}+\norm{S}$, where $\norm{\blank}$ denotes the cardinality of a set. We also borrow some convenient notation from \cite{neamtu2007delaunay}. In particular, given a configuration of $n\geq d+1$ points $A\df(a_1,\ldots,a_n)$ in $\mathbb{R}^d$ and a set of indices $I\subseteq [n]$ such that the points $(a_i)_{i\in I}$ are affinely independent, we denote by $\detp{I}$ the $(d+1)\times (d+1)$ determinant $\det((a_i,1)_{i\in I})$, with the rows ordered so that $\detp{I}>0$. Similarly, we denote by $\dets{I}{j}{k}$ the result of replacing the row corresponding to $(a_j,1)$ in $\detp{I}$ with $(a_k,1)$ in the same position. Notice that $\dets{I}{j}{k}$ is not necessarily positive. Similarly, for $x\in\mathbb{R}^d$, $\dets{I}{j}{x}$ is obtained by replacing the row $(a_j,1)$ in $\detp{I}$ with $(x,1)$. 

Let now $\mathcal{R}(A)\subset\mathbb{R}^d$ be any region obtained as the union of convex hulls of subsets of points in $A$. A {\em subdivision} $\mathcal{T}$ of $\mathcal{R}(A)$ is a collection of $d$-dimensional polytopes $\Delta$ with vertices in $A$ such that $\bigcup_{\Delta\in\mathcal{T}}\Delta=\mathcal{R}(A)$ and any two polytopes in $\mathcal{T}$ share a common face, possibly empty. If all the polytopes are simplices, then $\mathcal{T}$ is a {\em triangulation} of $\mathcal{R}(A)$.

\subsection{Multivariate splines}
\label{sec:back:splines}
Multivariate (unstructured) spline functions were introduced by Curry and Schoenberg \cite{curry1966polya}. The following useful recurrence formula was first derived by Micchelli \cite{micchelli1980constructive}. Given a configuration $A=(a_1,\ldots,a_n)$ of points in $\mathbb{R}^d$ and a subset $X\subseteq[n]$ of size $\norm{X}=k+d+1$, the normalized multivariate spline function $\M{x}{(a_i)_{i\in X}}$ can be defined for $x\in\mathbb{R}^d$ via the recursive expression
\begin{subequations}
\label{eqn:splinerec}
\begin{empheq}[left=\!\!\!\M{x}{(a_i)_{i\in X}}\df\empheqlbrace\,]{alignat=2}
&\dfrac{d!}{\detp{X}}\indicf{X}(x)&\;\mbox{ if }k=0,\label{eqn:splinerec-1}\\
&\frac{k+d}{k}\sum_{b\in B}\frac{\dets{B}{b}{x}}{\detp{B}}\M{x}{(a_i)_{i\in B\setminus\{b\}}}&\;\mbox{ otherwise},
\label{eqn:splinerec-2}
\end{empheq}
\end{subequations}
\noindent where $\indicf{X}(x)\df\indicf{\ch(\{a_i\}_{i\in X})}(x)$ is the indicator function of the convex hull of the points indexed by $X$, and $B$ is any subset $B\subseteq X$ with $\norm{B}=d+1$ such that the points $(a_b)_{b\in B}$ are affinely independent. If no such $B$ exists, then the affine rank of the points indexed by $X$ is less than $d+1$ and the spline, supported on a zero-measure set, is set to zero everywhere by continuity. The functions $\M{x}{\blank}$ are multivariate piecewise-polynomial functions of $x\in\mathbb{R}^d$ with regularity $C^{k-1}$ if all the points are affinely independent, and with reduced regularity otherwise.
Another useful expression, also derived in \cite{micchelli1980constructive}, is the {\em knot insertion} formula. If $\norm{X}\geq d+2$ (i.e., if $k\geq 1$), we can select another index $c\in X\setminus B$. We then have
\begin{equation}
\label{eqn:splinepar}
   \detp{B}\M{x}{(a_i)_{i\in X\setminus\{c\}}}= \sum_{b\in B}\dets{B}{b}{c}\M{x}{(a_i)_{i\in X\setminus\{b\}}}.
\end{equation}
Just like \eqref{eqn:splinerec-2} relates splines of order $k$ and $k-1$, allowing for a recurrent evaluation scheme, \eqref{eqn:splinepar} relates splines with the same order $k-1$. 

\subsection{Vector configurations and zonotopal tilings}
\label{sec:back:om}
We refer the reader to \cite{richter1994zonotopal} or \cite[Chapter~6]{ziegler2012lectures} for a thorough introduction to these combinatorial objects.

Let $A=(a_1,\ldots,a_n)$ be a configuration of points $a_i\in\mathbb{R}^d$, not necessarily affinely independent or even distinct, but which affinely span $\mathbb{R}^d$. For each point $a_i$, define its {\em projective lift} as $v_i\df (a_i,1)\in\mathbb{R}^{d+1}$, and let $V\df (v_1,\ldots,v_n)$ be the associated {\em vector configuration}.

Given two subsets $P$,~$Q\subset\mathbb{R}^d$, their {\em Minkowski sum} is defined as the set $P+Q\df \{x+y\in\mathbb{R}^d:\,x\in P,\,y\in Q\}$. The Minkowski sum of a set of segments is a special convex polytope known as a {\em zonotope}. There is a natural zonotope $Z(V)\subset\mathbb{R}^{d+1}$ associated to each point configuration $V$, defined as follows. For every index $i\in [n]$, define the segment $[0,v_i]\df\{\alpha_i v_i\in\mathbb{R}^{d+1}:0\leq\alpha_i\leq 1\}$. Then $Z(V)$ is given by the Minkowski sum
\begin{equation}
\label{eqn:zonodef}
    Z(V)\df\sum_{i=1}^n [0,v_i]
\end{equation}
Given two subsets of indices $I$,~$B\subseteq [n]$ with $I\cap B=\varnothing$, $\norm{B}=d+1$ and $\detp{B}>0$, the parallelepiped $\Pi_{I,B}\subset\mathbb{R}^{d+1}$ is defined as
\begin{equation}
\label{eqn:zonotile}
    \Pi_{I,B}\df\sum_{i\in I}v_i+\sum_{b\in B}[0,v_b]
\end{equation}
Notice that the $(d+1)$-dimensional volume of the tile $\vol^{d+1}(\Pi_{I,B})$ is equal to $\det(B)$, and that only $B$ determines the shape of $\Pi_{I,B}$, while $I$ simply shifts its position. A collection $\mathcal{P}$ of parallelepipeds $\Pi_{I,B}$ forming a polyhedral subdivision of $Z(V)$ is known as a {\em fine zonotopal tiling} of $Z(V)$ (see  \cite{bjorner1999oriented} or \cite[Chapter~6]{ziegler2012lectures}). An example is shown in Figure \ref{fig:zonotope}. Notice that the set $I$ of each tile $\Pi_{I,B}$ can be read off as the set of vectors in any shortest path connecting the origin to the base of the tile. In the present work, we call $\norm{I}$ the {\em order} of the tile $\Pi_{I,B}$, and we denote by $\mathcal{P}^{(k)}$ for any integer $k\geq 0$ the subset $\{\Pi_{I,B}\in\mathcal{P}:\norm{I}=k\}$.

The faces of a tile $\Pi_{I,B}$ are themselves parallelepipeds that are obtained by setting $\alpha_i$ equal to $0$ or $1$ in some of the segments $[0,v_b]$ of \eqref{eqn:zonotile}. Clearly, if $\Pi_{J,C}$ is a face of $\Pi_{I,B}$ then $C\subseteq B$ and $I\subseteq J\subseteq I\sqcup B$. If $\norm{C}=d$ then $\Pi_{J,C}$ is called a {\em facet} of $\Pi_{I,B}$. Since $\mathcal{P}$ is a subdivision, a facet is either shared between exactly two tiles of $\mathcal{P}$, or is an external facet of $Z(V)$. It is easily checked that two tiles $\Pi_{I,B}$ and $\Pi_{I^\prime,B^\prime}$ share a facet if and only if there are two indices $b\in B$, $b^\prime\in B^\prime$ such that $B\setminus\{b\}=B^\prime\setminus
\{b^\prime\}=B\cap B^\prime$ and either $I=I^\prime$, $I=I^\prime\sqcup\{b^\prime\}$, $I^\prime=I\sqcup\{b\}$ or $I\sqcup\{b\}=I^\prime\sqcup\{b^\prime\}$. The shared facet $\Pi_{J,C}$ then satisfies $C=B\cap B^\prime$ and $J=I\cup I^\prime$. 

Fine zonotopal tilings possess a number of remarkable properties. First, all such tilings of $Z(V)$ are simply different arrangements of the same set of tile shapes.
\begin{theorem}[Shephard \cite{shephard1974combinatorial}]
\label{thm:zonsub}
Every zonotope $Z(V)$ admits a fine zonotopal tiling, and all fine zonotopal tilings of $Z(V)$ have the same number of tiles, namely one full-dimensional tile for each maximal linearly independent subset of $V$.
\end{theorem}
Moreover, one can remove a point $a_i$ corresponding to an index $i\in [n]$ from $A$ and consider the corresponding zonotope $Z(V\setminus\{v_i\})$. Then, any tiling $\mathcal{P}$ of $Z(V)$ induces a tiling $\mathcal{P}_{[n]\setminus\{i\}}$ of $Z(V\setminus\{v_i\})$, or indeed of any zonotope built on a subset of $V$, as follows. 
\begin{lemma}
\label{lem:inducedtil}Let $\mathcal{P}$ be a fine zonotopal tiling of $Z(V)$. Then:
\begin{equation}
\label{eqn:indtil1}
    \mathcal{P}_{[n]\setminus\{i\}}\df \left\{\Pi_{I,B}\in\mathcal{P}:i\not\in I\sqcup B\right\}\sqcup\left\{\Pi_{I\setminus\{i\},B}:\Pi_{I,B}\in\mathcal{P},\,i\in I\right\}
\end{equation}
is a fine zonotopal tiling of $Z(V\setminus \{v_i\})$. Similarly, for any $Q\subseteq [n]$,
\begin{equation}
\label{eqn:indtil2}
    \mathcal{P}_{[n]\setminus Q}\df \left\{\Pi_{I\setminus Q,B}:\Pi_{I,B}\in\mathcal{P},\,B\cap Q=\varnothing\right\},
\end{equation}
is a fine zonotopal tiling of $Z(V\setminus \{v_q\}_{q\in Q})$.
\end{lemma}
\begin{proof}
For \eqref{eqn:indtil1}, see e.g. Proposition 4.3 of \cite{galashin2019higher}; \eqref{eqn:indtil2} follows from \eqref{eqn:indtil1} by repeated application.
\end{proof}

Since the tiles in $\mathcal{P}$ form a polyhedral subdivision of $Z(V)$, we can form its adjacency graph $\mathcal{G}$ by associating to each tile $\Pi_{I,B}$ a vertex in $\mathcal{G}$ and by connecting two tiles $\Pi_{I,B}$ and $\Pi_{I^\prime,B^\prime}$ with an edge if and only if the tiles share a facet.

\begin{figure}[htbp]
    \centering
    \subfloat{\includegraphics[width=0.42\linewidth]{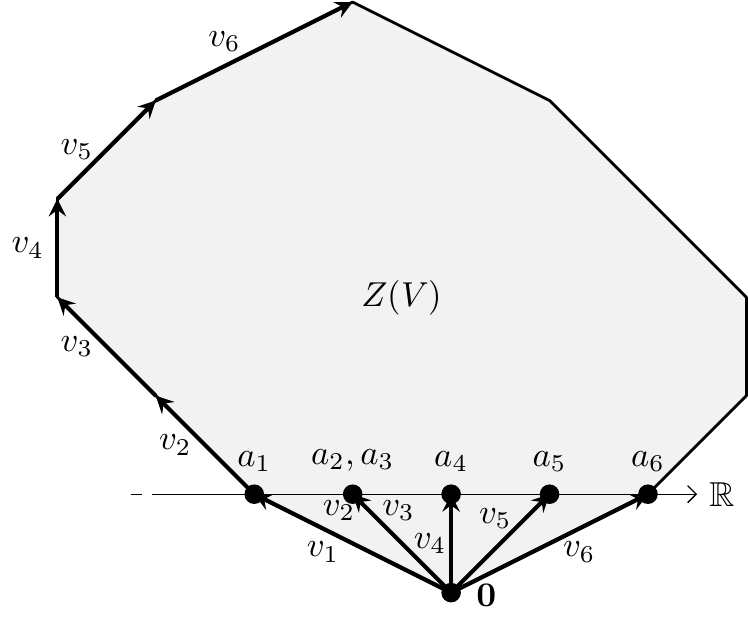}}\hspace{0.05\linewidth}%
    \subfloat{\includegraphics[width=0.50\linewidth]{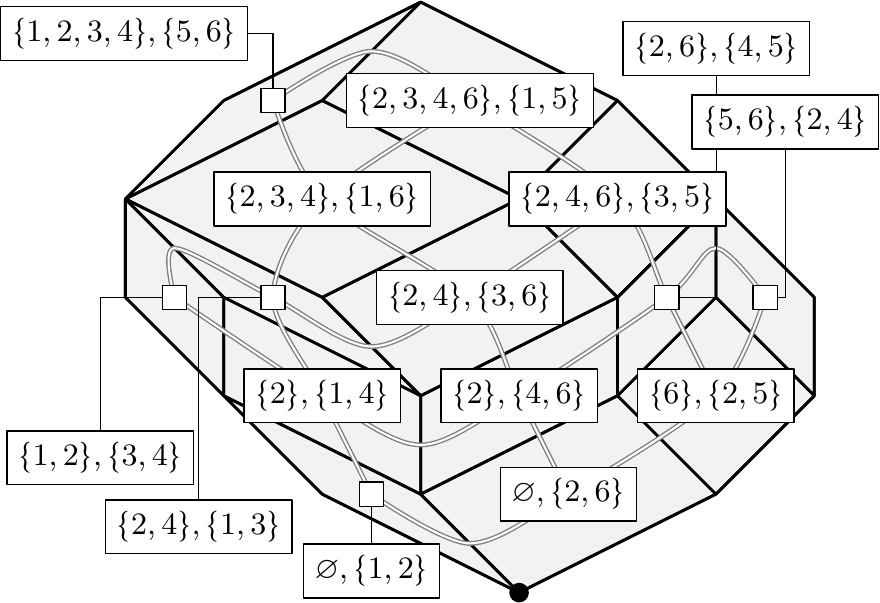}}
    \caption{{\em Left}: a point configuration $a_1,\ldots,a_6$ in $\mathbb{R}$, with $a_2=a_3$, their projective lifts $v_1,\ldots,v_6$ and the zonotope $Z(V)$. {\em Right}: a fine zonotopal tiling of $Z(V)$, the subsets $I,B$ associated to each tile $\Pi_{I,B}$, and the adjacency graph $\mathcal{G}$.}
    \label{fig:zonotope}
\end{figure}

\section{Polynomial-reproducing spline spaces}
\label{sec:sp}
The degree of approximation of a spline space is closely related to the maximal degree of polynomials it contains in its linear span \cite{de1990quasiinterpolants}. Such spaces are called {\em polynomial-reproducing}. Determining which spline spaces are polynomial-reproducing proved more challenging in $d>1$ than in the one-dimensional case. Many interesting spline spaces have been found on suitable triangulations and subdivisions (see e.g. \cite{lyche2017stable,bracco2016generalized}). We focus here on a recent approach by Neamtu \cite{neamtu2007delaunay} that is not based on a pre-existing subdivision. In his work, Neamtu showed that spline functions associated to Delaunay configurations of order $k$ form indeed a polynomial-reproducing spline space up to degree $k$. We introduce here briefly his results, before proposing a generalization.

First, let us recall the definition of the {\em polar form} of a polynomial (see e.g. \cite{ramshaw1989blossoms}):
\begin{definition}
\label{def:polform}
Let $k\geq 0$ and let $q(x)$, $x\in\mathbb{R}^d$, be a $d$-variate polynomial of degree at most $k$. Then there exists a unique function $Q(x_1,\ldots,x_k)$ of the $d$-dimensional variables $(x_1,\ldots,x_k)$ that is symmetric under permutation of its arguments, affine in each of them, and that agrees with $q$ on the diagonal, i.e., $Q(x,\ldots,x)=q(x)$. $Q$ is called the {\em polar form} of $q$.
\end{definition} 

Let $A$ be an infinite set of points in $\mathbb{R}^d$ in {\em general position}, i.e., where no subset of $d+1$ points is affinely dependent and no subset of $d+2$ points is co-spherical, and with no accumulation point. A {\em Delaunay configuration} $X_{I,B}$ of order $k\geq 0$ is any disjoint couple of sets $B$,~$I\subseteq [n]$ with $\norm{B}=d+1$, $\norm{I}=k$ such that the sphere circumscribed to the simplex $\Delta_B\df\ch(\{a_b\}_{b\in B})$ contains in its interior the points $\{a_i\}_{i\in I}$ and no other point of $A$. Notice that this definition depends crucially on the points being in general position. To each such configuration, we can associate through \eqref{eqn:splinerec} the $d$-variate spline function of order $k$
\begin{equation}
    \label{eqn:neamtuspline}
    \M{x}{X_{I,B}}\df\M{x}{\{a_i\}_{i\in I\sqcup B}}.
\end{equation}
Neamtu's result can be stated as follows:
\begin{theorem}[Neamtu \cite{neamtu2007delaunay}]
\label{thm:polyrepn}
Let $q(x)$ be a polynomial of degree at most $k$. Then, for all $x\in\mathbb{R}^d$,
\begin{equation}
    q(x)=\binom{k+d}{d}^{-1}\sum_{X_{I,B}\in D_k}Q((a_i)_{i\in I})\vol^d(\Delta_B)\M{x}{X_{I,B}},
\end{equation}
where $Q$ is the polar form associated to $q$ and the sum is extended to the set $D_k$ of Delaunay configurations of $A$ of order $k$.
\end{theorem}
Neamtu's result is based upon some strong assumptions on $A$, notably the infiniteness of the general position of points in $A$, which we are able to relax by using the combinatorial nature of zonotopal tilings to our advantage.

Let now $A=(a_1,\ldots,a_n)$ be any finite point configuration in $\mathbb{R}^d$. Assume that the affine span of the points in $A$ is the whole $\mathbb{R}^d$. Let $V$ be the associated vector configuration and $Z(V)$ its associated zonotope, as in Section \ref{sec:back:om}. Then the following, more general statement holds:

\begin{theorem}
\label{thm:polyrep}
Let $\mathcal{P}$ be a fine zonotopal tiling of $Z(V)$, let $0\leq k\leq n-d-1$ and let $\mathcal{P}^{(k)}\df \{\Pi_{I,B}\in\mathcal{P}:\,\norm{I}=k\}$. Each tile $\Pi_{I,B}\in\mathcal{P}^{(k)}$ can be associated via \eqref{eqn:splinerec} to the $d$-variate spline of degree $k=\norm{I}$
\begin{equation}
\label{eqn:splinedef}
   \M{\blank}{\Pi_{I,B}}\df \M{\blank}{(a_i)_{i\in I\sqcup B}}.
\end{equation}
Then, for any polynomial $q(x)$ of degree at most $k$,
\begin{equation}
\label{eqn:polyrep}
	q(x)=\frac{k!}{(k+d)!}\sum_{\Pi_{I,B}\in\mathcal{P}^{(k)}}\!\!\!Q((a_i)_{i\in I})\vol^{d+1}(\Pi_{I,B})\M{x}{\Pi_{I,B}}\mbox{ for } x\in\ch_k(A),
\end{equation}
where $Q$ is the polar form of $q(x)$ and
\begin{equation}
\label{eqn:chpa}
    \ch_k(A)=\bigcap_{\substack{S\subseteq [n]\\ \norm{S}=n-k}}\ch(\{a_i\}_{i\in S})
\end{equation}
is the intersection of the convex hulls of all subconfigurations of $A$ of size $n-k$.
\end{theorem}

The generalization with respect to Neamtu's result is twofold. First, for a given point configuration $A$, many different fine zonotopal tilings of $Z(V)$ can be constructed. Each tiling then yields a family of polynomial-reproducing spline spaces for all degrees up to $n-d-1$. In fact, Delaunay configurations can be seen as a special case of this construction, as discussed in the next section.

A second generalization is that the point configuration $A$ is allowed to contain affinely dependent subsets and repeated points. In this case, some of the spline functions have reduced regularity \cite{micchelli1980constructive}, and thus the spline spaces that can be constructed in this way are more generic. Observe that, if all the vertices of $\ch(A)$ are repeated at least $k+1$ times in $A$, then $\ch_k(A)=\ch(A)$. We obtain therefore a multivariate generalization of the behavior of {\em clamped} (also called {\em open}) knot vectors in one dimension:
\begin{corollary}
\label{cor:polyrepb}
Assume that each vertex of $\ch(A)$ is repeated at least $k+1$ times in $A$. Then, in the same conditions as Theorem \ref{thm:polyrep}, the splines $\M{\blank}{\Pi_{I,B}}$ for $\Pi_{I,B}\in\mathcal{P}^{(k)}$ reproduce polynomials up to order $k$ on the whole $\ch(A)$.
\end{corollary}
\noindent This is a highly desirable property in view of practical applications, as it allows the definition of boundary conditions.

\subsection{Proof of Theorem \ref{thm:polyrep}}
Neamtu's original proof of the fact that splines associated to Delaunay configurations are polynomial-reproducing (Theorem 4.1 of \cite{neamtu2007delaunay}) rests on a crucial structural property regarding neighbouring pairs of configurations, namely the {\em edge matching} property proved in Proposition 2.1 of \cite{neamtu2007delaunay}. This property underpins also other formulations such as the algorithmic generalization proposed by Liu and and Snoeyink \cite{liu2008computations} and the geometric description of Schmitt in terms of families of convex Jordan curves \cite{schmitt2019bivariate}. We prove hereafter that a similar property also holds for zonotopal tilings.

\begin{proposition}
\label{prop:faces}
Let $\Pi_{J,C}$ be a facet of a tile $\Pi_{I,B}\in\mathcal{P}$, with $\norm{J}=k$. Then $\norm{I}=k$ or $\norm{I}=k-1$, and exactly one of the following is true:
\begin{enumerate}[label=(\roman*)]
\item $\Pi_{J,C}$ is shared between $\Pi_{I,B}$ and exactly another tile $\Pi_{I^\prime,B^\prime}\in\mathcal{P}$, with either $\norm{I^\prime}=k$ or $\norm{I^\prime}=k-1$. Moreover, if $\{b\}=B\setminus B^\prime$ and $\{b^\prime\}=B^\prime\setminus B$, the two points $a_b$ and $a_{b^\prime}$ are separated by the hyperplane $H\df\aff(\{a_c\}_{c\in C})$ if and only if $\norm{I}=\norm{I^\prime}$;\label{item:faces1}
\item there exists an index $b\in B$ such that, for a suitable orientation of the hyperplane $H\df\aff(\{a_c\}_{c\in C})$, the points $\{a_i\}_{i\in I}$ are in the positive closed halfspace of $H$, the points $\{a_i\}_{i\in\overline{I\sqcup B}}$ are in the negative closed halfspace of $H$, and $a_b$ is in the positive open halfspace of $H$ if $b\in J$ and in the negative open halfspace of $H$ if $b\not\in J$.\label{item:faces2}
\end{enumerate}
\end{proposition}
\begin{proof}
A facet $\Pi_{J,C}$ of a tile $\Pi_{I,B}$ is obtained by choosing an index $b\in B$ and setting the corresponding coefficient $\alpha_b$ of segment $[0,v_b]$ in \eqref{eqn:zonotile} to either $0$, in which case $J=I$, or $1$, in which case $J=I\sqcup\{b\}$. Thus, $k\df\norm{J}=\norm{I}$ or $k\df\norm{J}=\norm{I}+1$. Since the tiles in $\mathcal{P}$ form a subdivision of $Z(V)$, $\Pi_{J,C}$ is either a shared facet between $\Pi_{I,B}$ and exactly one other tile $\Pi_{I^\prime,B^\prime}$, or is a boundary facet of $Z(V)$. 

In the first case, $C=B\cap B^\prime$, and the previous argument also implies that either $J=I^\prime$ or $J=I^\prime\sqcup\{b^\prime\}$, with $\{b^\prime\}=B^\prime\setminus B$ and $\{b\}=B\setminus B^\prime$, and thus $\norm{I^\prime}=k$ or $\norm{I^\prime}=k-1$. Since both parallelepipeds are convex polytopes, their interiors are separated by the hyperplane spanned by their common facet, and we can choose a nonzero vector $N\in\mathbb{R}^{d+1}$, normal to the facet, satisfying $\scal{v_c}{N}=0$ for all $c\in C= B\cap B^\prime$, and
\begin{equation}
    \label{eqn:sharedface}
    \scal{z-z^\prime}{N}\geq 0
\end{equation}
for all $z\in\Pi_{I,B}$ and $z^\prime\in\Pi_{I^\prime, B^\prime}$. Notice that necessarily $\scal{v_b}{N}\neq 0$ and $\scal{v_{b^\prime}}{N}\neq 0$, since the vectors in $B$ and $B^\prime$ must be linearly independent. The case $\norm{I}=\norm{I^\prime}$ corresponds to either $I=I^\prime$ or $I\sqcup\{b\}=I^\prime\sqcup\{b^\prime\}$. If $I=I^\prime$, then setting $(z,z^\prime)=(v_b+\sum_{i\in I}v_i,\sum_{i\in I^\prime} v_i)$ in \eqref{eqn:sharedface} yields $\scal{v_b}{N}>0$, while choosing $(z,z^\prime)=(\sum_{i\in I}v_i,v_{b^\prime}+\sum_{i\in I^\prime} v_i)$ yields $\scal{v_{b^\prime}}{N}<0$. Thus,
\begin{equation}
    \sign(\scal{v_b}{N})=-\sign(\scal{v_{b^\prime}}{N}).
\end{equation}
If $I\sqcup\{b\}=I^\prime\sqcup\{b^\prime\}$, the same choices of $(z,z^\prime)$ lead to the same conclusion. The case $\norm{I}\neq\norm{I^\prime}$ is very similar, since it implies either $I=I^\prime\sqcup\{b^\prime\}$ or $I\sqcup\{b\}=I^\prime$. In both cases, plugging the couples $(z,z^\prime)=(\sum_{i\in I} v_i,\sum_{i\in I^\prime} v_i)$ and $(z,z^\prime)=(v_b+\sum_{i\in I} v_i,v_{b^\prime}+\sum_{i\in I^\prime} v_i)$ 
in \eqref{eqn:sharedface} leads to 
\begin{equation}
    \sign(\scal{v_b}{N})=\sign(\scal{v_{b^\prime}}{N}).
\end{equation}
Thus, the hyperplane $H=\{x\in\mathbb{R}^d:\scal{N}{(x,1)}=0\}$ satisfies the first part of the proposition.
 
Suppose now that $\Pi_{J,C}$ is a boundary facet of $Z(V)$. Since $Z(V)$ is a convex polytope, all points $z\in Z(V)$ lie in the same closed halfspace of $\Pi_{J,C}$, and we can choose a nonzero vector $N\in\mathbb{R}^{d+1}$, normal to $\Pi_{J,C}$, so that $\scal{v_c}{N}=0$ for all $c\in C$ and
\begin{equation}
    \label{eqn:extface}
    \langle z-\sum_{j\in J} v_j,N\rangle\leq 0
\end{equation}
for all $z\in Z(V)$. Plugging into \eqref{eqn:extface}, respectively, $z=v_e+\sum_{j\in J}v_j$ with $e\not\in J$ and $z=\sum_{j\in J,j\neq f}v_j$ with $f\in J$ shows that
\begin{equation}
    \scal{v_c}{N}=0,\,\scal{v_e}{N}\leq 0,\scal{v_f}{N}\geq 0
\end{equation}
for all $c\in C$, $e\not\in J$ and $f\in J$. Moreover, as before, $\scal{v_b}{N}\neq 0$, otherwise the vectors in $B$ would be linearly dependent. Therefore, $\scal{v_b}{N}>0$ if $b\in J$, and $\scal{v_b}{N}<0$ if $b\not\in J$. Since $I\subseteq J\subseteq I\sqcup B$, the hyperplane $H=\{x\in\mathbb{R}^d:\scal{N}{(x,1)}=0\}$ satisfies the second part of the proposition. 
\end{proof}

Alternative \ref{item:faces1} of Proposition \ref{prop:faces} corresponds exactly to (a generalization of) essential and non-essential faces between Delaunay configurations that are described in Proposition 2.1 of \cite{neamtu2007delaunay}. However, in Proposition \ref{prop:faces} above, the underlying point set $A$ is finite, leading to the additional case \ref{item:faces2}. Notice that the points are not required to be in general position, and can even be repeated multiple times in $A$.

Armed with this result, we are ready to establish the polynomial reproduction property for spline functions associated to $\mathcal{P}$. The proof is similar to that of Theorem 4.1 of \cite{neamtu2007delaunay}; nonetheless, we give here the full derivation in order to point out the contribution of boundary facets. We start by proving the case $k=0$. 
\begin{proposition}
\label{prop:triangulation}
Let $\mathcal{P}^{(0)}\df\{\Pi_{\varnothing,B}\in\mathcal{P}\}$. Then the set of simplices $\mathcal{T}^{(0)}=\{\Delta_B\df\ch(\{a_b\}_{b\in B}):\,\Pi_{\varnothing,B}\in\mathcal{P}^{(0)}\}$ triangulates $\ch(A)$.
\end{proposition}
\begin{proof}
The proof can be derived from equivalent statements in \cite{santos2001realizable,santos2002triangulations} or \cite[Chapter~9]{bjorner1999oriented}. We give here a short direct proof for convenience. First, for any tile $\Pi_{\varnothing,B}\in\mathcal{P}^{(0)}$, the points $\{a_b\}_{b\in B}$ are affinely independent, and thus all the simplices in $\mathcal{T}^{(0)}$ are non-degenerate. Let $\Pi_{\varnothing,B^\prime}$ be a distinct tile in $\mathcal{P}^{(0)}$, and assume that there is a positive linear dependency
\begin{equation}
\label{eqn:posdep}
    \sum_{b\in B} \beta_b v_b + \sum_{b^\prime\in B^\prime}\gamma_{b^\prime} (-v_{b^\prime}) = 0
\end{equation}
with $\beta_b,\gamma_{b^\prime}>0$. If we define $C\df\max\left(\{\beta_b\}_{b\in B}\cup\{\gamma_{b^\prime}\}_{b^\prime\in B^\prime}\right)$, then the point $z\df\sum_{b\in B} \beta_b/C\, v_b= \sum_{b^\prime\in B^\prime}\gamma_{b^\prime}/C\,  v_{b^\prime}$ lies in the interior of both $\Pi_{\varnothing, B}$ and $\Pi_{\varnothing, B^\prime}$, which is impossible since $\mathcal{P}$ is a polyhedral subdivision. Therefore, there cannot exist any positive linear dependency \eqref{eqn:posdep} and, by Stiemke's Lemma \cite{stiemke1915positive}, there must be a vector $N\in\mathbb{R}^{d+1}$ with $\scal{N}{v_b}\geq 0$ for all $b\in B$ and $\scal{N}{v_{b^\prime}}\leq 0$ for all $b^\prime\in B^\prime$. The corresponding hyperplane $\{x\in\mathbb{R}^d:\scal{N}{(x,1)}=0\}$ separates $\Delta_B$ and $\Delta_{B^\prime}$, proving that they have disjoint interiors.

Finally, let $\Pi_{\varnothing,C}$ be the facet of $\Pi_{\varnothing,B}$ obtained by setting, for a single $b\in B$, the coefficient $\alpha_b$ of the segment $[0,v_b]$ in \eqref{eqn:zonotile} equal to zero. Then $F_C\df\ch(\{a_c\}_{c\in C})$ is a $(d-1)$-dimensional face of $\Delta_B$. By Proposition \ref{prop:faces}, either there is a unique tile $\Pi_{\varnothing,B^\prime}$ with $\norm{B\cap B^\prime}=d$, i.e., there is exactly one distinct simplex $\Delta_{B^\prime}$ in $\mathcal{T}^{(0)}$ sharing $F_C$ with $\Delta_B$, or $F_C$ lies on a hyperplane that does not contain any point of $A$ on its positive side, and therefore belongs to the boundary of $\ch(A)$. This completes the proof.
\end{proof}
The indicator functions of simplices in $\mathcal{T}^{(0)}$ correspond exactly to degree-zero splines via \eqref{eqn:splinerec-1}. Proposition \ref{prop:triangulation} then provides the root of the recurrence in the following proof.
\begin{proof}[Proof of Theorem \ref{thm:polyrep}]
Similarly to the proof of Theorem 4.1 in \cite{neamtu2007delaunay}, we simply have to prove that the expression
\begin{equation}
    \label{eqn:polyrepbase}
     \sum_{\Pi_{I,B}\in\mathcal{P}^{(k)}}\!\!Q((a_i)_{i\in I})\vol^{d+1}(\Pi_{I,B})\M{x}{\Pi_{I,B}}
\end{equation}
can be rewritten in terms of the tiles in $\mathcal{P}^{(k-1)}$ as
\begin{equation}
\label{eqn:polyrep2}
   \frac{k+d}{k}\sum_{\Pi_{I^\prime,B^\prime}\in\mathcal{P}^{(k-1)}}\!\!Q((a_i)_{i\in I^\prime},x)\vol^{d+1}(\Pi_{I^\prime,B^\prime})\M{x}{\Pi_{I^\prime,B^\prime}}.
\end{equation}
In fact, iterating until $k=0$ directly leads to the expression
\begin{equation}
\label{eqn:polyreprec2}
    \binom{k+d}{k}\sum_{\Pi_{\varnothing,B}\in\mathcal{P}^{(0)}}\!\!Q(x,\ldots,x)\vol^{d+1}(\Pi_{\varnothing,B^\prime})\M{x}{\Pi_{\varnothing,B^\prime}},
\end{equation}
which is simply equal to $(k+d)!/k!\,q(x)$ thanks to \eqref{eqn:splinerec-1}, the definition of polar form (Definition \ref{def:polform}), and the fact that the simplices defined by splines in $\mathcal{P}^{(0)}$ triangulate $\ch(A)$ (Proposition \ref{prop:triangulation}).

In order to prove that \eqref{eqn:polyrepbase} is equal to \eqref{eqn:polyrep2}, similarly to \cite{neamtu2007delaunay}, we first apply the spline recurrence formula \eqref{eqn:splinerec-2} to \eqref{eqn:polyrepbase}, obtaining
\begin{equation}
\label{eqn:sump}
    \frac{k+d}{k}\!\!\!\sum_{\Pi_{I,B}\in\mathcal{P}^{(k)}}\!\!\!Q((a_i)_{i\in I})\!\!\sum_{b\in B}\!\!\dets{B}{b}{x}\M{x}{\Pi_{I,B\setminus\{b\}}},
\end{equation}
since $\vol^{d+1}(\Pi_{I,B})=\det(B)$. We can associate every term in \eqref{eqn:sump} with a facet $\Pi_{I,B\setminus\{b\}}$ of $\mathcal{P}$. Following Proposition \ref{prop:faces}, there are three possibilities:
\begin{enumerate}[label=(\roman*)]
\item The facet is shared with exactly one other tile $\Pi_{I^\prime,B^\prime}\in\mathcal{P}^{(k)}$, with $I^\prime=I$, $B^\prime\setminus \{b^\prime\}=B\setminus \{b\}=B\cap B^\prime$ for some $b^\prime\in B^\prime$, and with $a_b$ and $a_{b^\prime}$ lying on opposite sides of $H\df\aff(\{v_i\}_{i\in B\cap B^\prime})$. Therefore $\dets{B}{b}{x}=-\dets{B^\prime}{b^\prime}{x}$, and the two corresponding terms in the sum cancel each other;
\item The facet is shared with exactly one other tile $\Pi_{I^\prime,B^\prime}\in\mathcal{P}^{(k-1)}$, with $I^\prime\sqcup\{b^\prime\}=I$, $B^\prime\setminus \{b^\prime\}=B\setminus\{b\}=B\cap B^\prime$ for some $b^\prime\in B^\prime$, and with $a_b$ and $a_{b^\prime}$ lying on the same side of $H\df\aff(\{a_i\}_{i\in B\cap B^\prime})$. After noticing that $I\sqcup B\setminus\{b\}=I^\prime\sqcup B^\prime$, the corresponding term in \eqref{eqn:sump} can be rewritten as
\begin{equation}
\label{eqn:efaces}
    \frac{k+d}{k}Q((a_i)_{i\in I^\prime\sqcup\{b^\prime\}})\dets{B^\prime}{b^\prime}{x}\M{x}{(a_i)_{i\in I^\prime\sqcup B^\prime}}.
\end{equation}
\item The facet lies on the boundary of $Z(V)$. In this case the hyperplane $H\df\aff(\{a_i\}_{i\in B\setminus\{b\}})$ contains all the points $\{a_i\}_{i\in I\sqcup B\setminus\{b\}}$ in its positive closed halfspace, out of which at most $\norm{I}=k$ are in its positive open halfspace. All other points of $A$ lie in its negative closed halfspace. Consequently, if $x$ is in the interior of $\ch_k(A)$, then necessarily $x\not\in\ch(\{a_i\}_{i\in I\sqcup B\setminus\{b\}})$ and therefore
\begin{equation}
    \M{x}{\Pi_{I,B\setminus\{b\}}}=\M{x}{(a_i)_{i\in I\sqcup B\setminus\{b\}}}=0.
\end{equation}
\end{enumerate}

Focusing now on \eqref{eqn:polyrep2}, and again similarly to \cite{neamtu2007delaunay}, we rewrite $x$ in barycentric coordinates with respect to the simplex $\ch(\{a_{b^\prime}\}_{b^\prime\in B^\prime})$ as 
\begin{equation}
\label{eqn:xbary}
    x=\sum_{b^\prime\in B^\prime}\frac{\dets{B^\prime}{b^\prime}{x}}{\detp{B^\prime}}a_{b^\prime},
\end{equation}
and since $Q$ is multiaffine and $\vol^{d+1}(\Pi_{I^\prime,B^\prime})=\detp{B^\prime}$, using \eqref{eqn:xbary}, we can rewrite \eqref{eqn:polyrep2} as
\begin{equation}
\label{eqn:sump2}
    \frac{k+d}{k}\sum_{\Pi_{I^\prime,B^\prime}\in\mathcal{P}^{(k-1)}}\!\!\!\!\!\M{x}{\Pi_{I^\prime,B^\prime}}\sum_{b^\prime\in B^\prime}Q((a_i)_{i\in I^\prime\sqcup\{b^\prime\}})\dets{B^\prime}{b^\prime}{x}.
\end{equation}
Similarly as before, by Proposition \ref{prop:faces}, we can associate each term in \eqref{eqn:sump2} with a facet $\Pi_{I^\prime\sqcup\{b^\prime\},B^\prime\setminus\{b^\prime\}}$ of $\mathcal{P}$. If such a facet is shared with exactly one other tile $\Pi_{I,B}\in\mathcal{P}^{(k-1)}$, then it appears twice in the sum, and the two contributions cancel each other since $I^\prime\sqcup\{b^\prime\}=I\sqcup\{b\}$, $I\sqcup B=I^\prime\sqcup B^\prime$ and $a_b$, $a_{b^\prime}$ are separated by $H\df\aff(\{a_i\}_{i\in B\cap B^\prime})$. Terms corresponding to facets on the boundary of $Z(V)$ again do not contribute to the sum, since the corresponding hyperplane $H\df\aff(\{a_i\}_{i\in B^\prime\setminus\{b^\prime\}})$ separates at most the $k$ points in $I^\prime\sqcup\{b^\prime\}$ from the other $n-k$ points of $A$, and since $b^\prime\not\in I^\prime$, the points $\{a_i\}_{i\in I^\prime\sqcup B^\prime}$ either lie on $H$ or on the positive side of $H$. Thus, if $x\in\ch_k(A)$, we have once more
\begin{equation}
    \M{x}{\Pi_{I^\prime,B^\prime}}=\M{x}{(a_i)_{i\in I^\prime\sqcup B^\prime}}=0.
\end{equation}
The remaining terms correspond to facets shared with exactly one other tile $\Pi_{I,B}\in\mathcal{P}^{(k)}$, and they are equal to the terms \eqref{eqn:efaces}, completing the proof.
\end{proof}

Two examples of families of spline spaces associated to fine zonotopal tilings are shown in Figure \ref{fig:zonobase}.

\begin{figure}[htbp]
    \centering
    \subfloat{\includegraphics[width=0.40\linewidth]{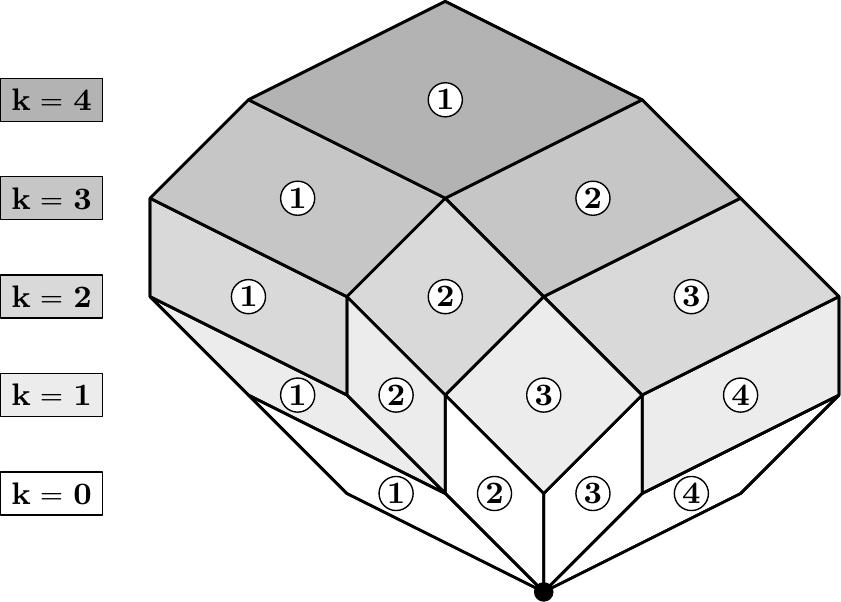}}\hspace{0.05\textwidth}%
    \subfloat{\includegraphics[width=0.40\linewidth]{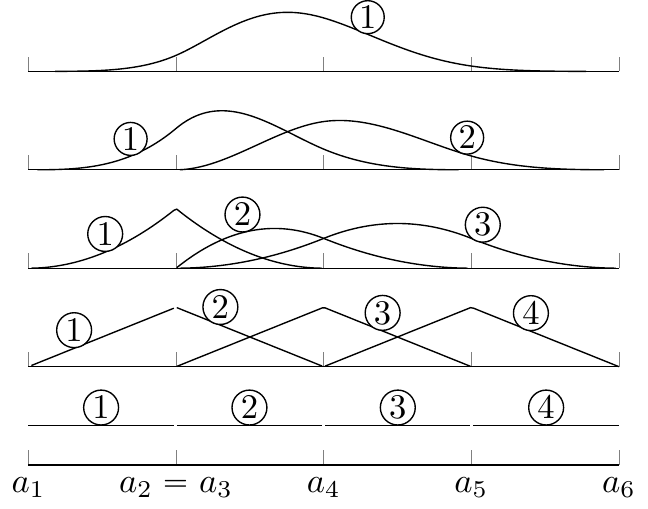}}
    \newline
    \subfloat{\includegraphics[width=0.40\linewidth]{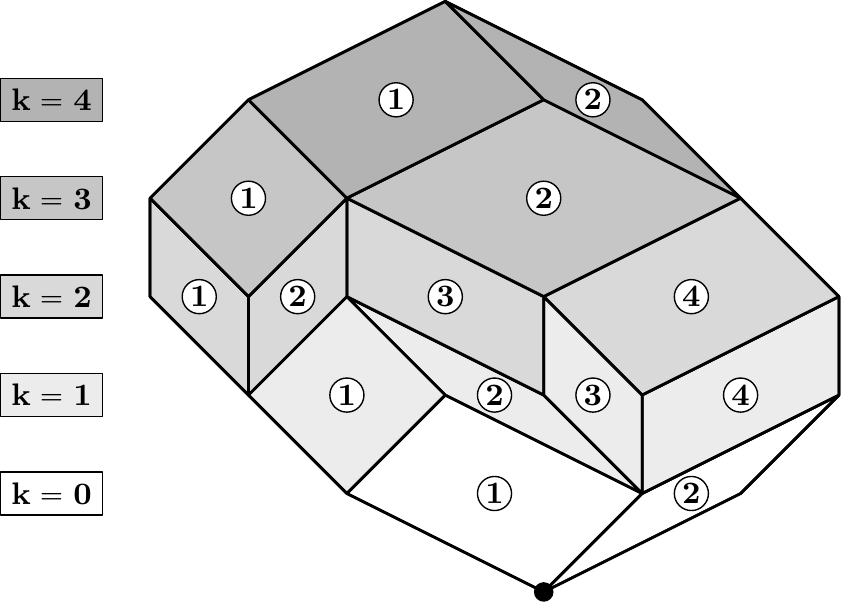}}\hspace{0.05\textwidth}%
    \subfloat{\includegraphics[width=0.40\linewidth]{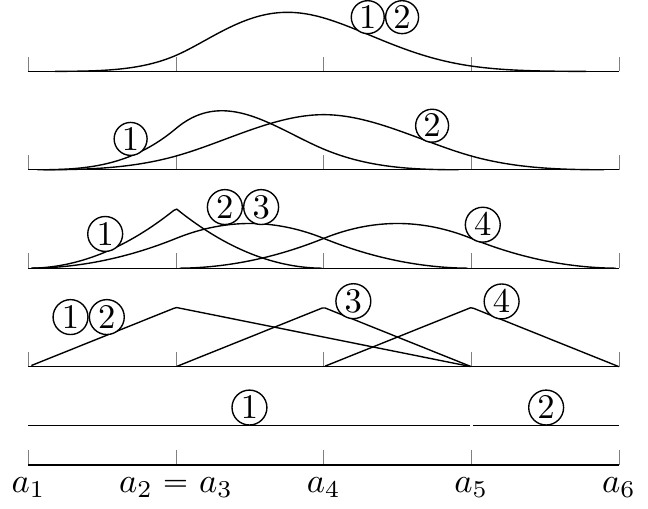}}
    \caption{Two possible fine zonotopal tilings of $Z(V)$ for the point configuration of Figure \ref{fig:zonotope} and their associated spline spaces of degrees $k=0,\ldots,4$ for the standard one-dimensional B-spline basis ({\em top}) and an alternative tiling ({\em bottom}).}
    \label{fig:zonobase}
\end{figure}

\subsection{Spline space construction}

Algorithms for the construction of Delaunay configurations (or, rather, their dual higher-order Voronoi diagrams) have been known for some time \cite{lee1982k}. In the two-dimensional case, Liu and Snoeyink \cite{liu2007quadratic,liu2008computations} have leveraged these results to propose an algorithm capable of iteratively constructing a large family of generalized Delaunay configurations of $A$ with any order $k\geq 0$, each yielding a set of polynomial-reproducing spline spaces. Their algorithm is based on the concept of the order-$k$ {\em centroid triangulation} \cite{schmitt1991delaunay,schmitt1998order,liu2008computations,el2011centroid}, which is a triangulation of the point set $A^{(k)}$ whose elements are the averages of $k$-element subsets of $A$. The order-$1$ centroid triangulation is simply an (arbitrary) triangulation of $A$, and an order-$k$ centroid triangulation is obtained from an order-$(k-1)$ centroid triangulation by a subdivision of the polygonal neighborhood of every vertex (its {\em link region}), with complete freedom in the choice of triangulation for each polygon. Every triangle obtained in this way is then assocated to a spline function of degree $k$.

In the two-dimensional case, this algorithm has been proven to converge for degrees $k\leq 3$ \cite{liu2008computations} and later for all degrees $k\geq 0$ by Schmitt \cite{schmitt2019bivariate}. However, one major hurdle for the extension to dimensions $d>2$ lies in the existence of non-convex regions that do not admit any triangulation, such as Schönhardt's polyhedron \cite{schonhardt1928zerlegung}. If such a region is encountered, the algorithm cannot continue, and there is no known condition under which the link regions are all guaranteed to be triangulable. Moreover, the case of affinely dependent and/or repeated points is excluded from the proofs and treated with symbolic perturbation, which creates ambiguous cases and does not allow to extend the proofs of convergence easily. This problem becomes even harder to address as the number of space dimensions grows. 

Given a fine zonotopal tiling $\mathcal{P}$ of $Z(V)$, we prove in this section that there exists a construction algorithm similar to Liu and Snoeyink's, with a suitable choice of triangulations, that is able to iteratively construct $\mathcal{P}$. This result rests on a natural definition of the {\em link region} $\mathcal{R}(I)$ associated to each subset $I\subset[n]$ (Definition \ref{def:reglink}), which generalizes naturally Liu and Snoeyink's notion of vertex link. 

\subsection{Relationship with centroid triangulations} 
Denoting by $H_r$ the hyperplane $H_r\df \{x\in\mathbb{R}^{d+1}:\,x_{d+1}=r\}$, the intersection
\begin{equation}
\label{eqn:ksetpoly}
Q^{(r)}\df Z(V)\cap H_r  
\end{equation}
corresponds to the set $Q^{(r)}\df\{\sum_{v\in B} [0,v_b]:\,B\subseteq [n],\,\norm{B}=r\}$, which is just the convex hull of the points $V^{(r)}\df\{\sum_{a_i\in B} (a_i,1),\,B\subseteq A,\,\norm{B}=r\}$. The region $Q^{(r)}$ is also known as (a multiple of) the {\em r-set polytope} of $A$ \cite{edelsbrunner1997cutting,schmitt2006k}. Just as vectors in $V$ can be interpreted projectively as points in $A$, vectors in $V^{(r)}$ can be projectively reduced to the set $A^{(r)}$ of all possible averages of $r$ points in $A$. The intersection $\mathcal{P}\cap H_r$ of a zonotopal tiling of $Z(V)$ with $H_r$ then produces a subdivision of $V^{(r)}$ \cite{olarte2019hypersimplicial,galashin2019higher} with (projective) vertices in $A^{(r)}$, which corresponds to a {\em centroid subdivision} in the sense of \cite{schmitt1991delaunay,schmitt1998order,liu2008computations,el2011centroid}. 

According to \eqref{eqn:zonotile}, the intersection of a tile $\Pi_{I,B}$, $\norm{I}=k$ with the hyperplane $H_r$ is an affine transformation of the hypersimplex $\Delta_{d+1,r-k}$, which has a positive dimension if and only if  $k<r<k+d+1$. Translated in the language of spline spaces, this means that the cells in the $r$-th centroid subdivision induced by $\mathcal{P}$ are slices of tiles associated via \eqref{eqn:splinedef} to the basis splines
\begin{equation}
\label{eqn:splinesk}
    \mathcal{SP}^{(r)}\df\{\M{\blank}{\Pi_{I,B}},\,r-d-1<k\df\norm{I}<r\}.
\end{equation}
For $d=2$, only two types of cells appear in each $r$-th centroid triangulation for $r>1$, corresponding to splines of degree $k=r-1$ and $k=r-2$. The corresponding hypersimplices $\Delta_{3,1}$ and $\Delta_{3,2}$ are just triangles, and therefore the subdivision is a so-called bicolored triangulation. This fact is widely known in the context of centroid triangulations \cite{lee1982k,schmitt1991delaunay,schmitt1998order,liu2008computations,el2011centroid}, where the corresponding triangles are called type-I and type-II triangles, respectively. In dimension $d>2$, the induced subdivision is no longer a triangulation, and the splines of all orders $r-d+1\leq k\leq r-2$ appear in the $r$-th centroid subdivision as hypersimplices, e.g., octahedra for $d=3$, $k=r-2$.

\subsection{Link regions} 
%Let us now return to  Liu and Snoeyink's algorithm. At every step, the algorithm builds a polygonal line, called the {\em link}, associated to every vertex in the $k$-th centroid triangulation. It can be easily checked, though we will not do it explicitly here, that the region bounded by the link associated to a configuration $I\subseteq [n]$, $\norm{I}=k$ can be obtained by taking the region $\bar{\mathcal{R}}$ composed by the boundary vertices $X_{B^\prime}$ of all configurations $X$ of order $r<k$ with $I\cap B^\prime=\varnothing$ and $I^\prime\subseteq I$, and taking its complement $\mathcal{R}$. The boundary of $\mathcal{R}$ corresponds to the vertex link of $I$. Motivated by this consideration, w
We define the {\em link region} of a subset $I\subseteq [n]$ as follows:
\begin{definition}
\label{def:reglink}
Given a fine zonotopal tiling $\mathcal{P}$ of $Z(V)$ and a subset $Q\subseteq [n]$, $\norm{Q}=k$, the regions $E^{(r)}(Q)$, $r\geq 0$, are defined as the union of simplices
\begin{equation}
\label{eqn:reglink2}
    E^{(r)}(Q)\df \bigcup_{\Pi_{I,B}\in\mathcal{E}^{(r)}(Q)}\ch(\{a_b\}_{b\in B}),    \end{equation}
with
\begin{equation}
    \mathcal{E}^{(r)}(Q)\df \left\{\Pi_{I,B}\in\mathcal{P}^{(r)}:\, B\cap Q=\varnothing,\, I\subseteq Q\right\}. \label{eqn:reglink3}
\end{equation}
The {\em link region} $\mathcal{R}(Q)$ of $Q$ is defined as $\mathcal{R}(Q)\df E^{(k)}(Q)$.
\end{definition}
An example of link region, and its relation to the regions \eqref{eqn:reglink2}, is shown in Figure \ref{fig:reglink}. Notice that $\mathcal{E}^{(k)}(Q)=\{\Pi_{I,B}\in\mathcal{P}:\,I=Q\}$ and that $\mathcal{E}^{(r)}(Q)=\varnothing$ for $r>k$.  It can be easily checked, though we will not do it explicitly here, that in two dimensions the above defined link region coincides with the interior of a vertex link as used in \cite{liu2007quadratic,liu2008computations,schmitt2019bivariate}.
However, Definition \ref{def:reglink} is more straightforward, more general, and can be applied to all point configurations in any dimension, allowing to easily prove some important properties, as we do presently.

\begin{proposition}
\label{prop:linktiling}
For any subset $Q\subseteq [n]$, define 
\begin{equation}
    \label{eqn:convq}
       \ch_Q(A)\df \ch(\{a_i\}_{i\not\in Q})
    \end{equation}
and let $r\geq 0$. Then, the following holds:
\begin{enumerate}[label=(\roman*)]
    \item The set of simplices $\mathcal{T}^{(r)}(Q)\df\{\ch(\{a_b\}_{b\in B}):\Pi_{I,B}\in\mathcal{E}^{(r)}(Q)\}$ forms a triangulation of $E^{(r)}(Q)$;\label{enum:linktiltri}
    \item The regions $E^{(r)}(Q)$ form a subdivision of $\ch_Q(A)$;\label{enum:linktilcover}
    \item The union of all simplices $\bigcup_{r\geq 0} \mathcal{T}^{(r)}(Q)$ triangulates $\ch_Q(A)$;\label{enum:linktiltot}
    \item The simplices $\mathcal{T}^{(k)}(Q)$ triangulate the link region $\mathcal{R}(Q)$.\label{enum:linktilreg}
    \end{enumerate}
\end{proposition}
\begin{proof}
Obviously, \ref{enum:linktiltri} implies \ref{enum:linktilreg} via Definition \ref{def:reglink}. Notice also that \ref{enum:linktiltot} implies both \ref{enum:linktilcover} and \ref{enum:linktiltri}, since it is clear from \eqref{eqn:reglink3} that $\mathcal{E}^{(r)}(Q)\cap\mathcal{E}^{(s)}(Q)=\varnothing$ if $r\neq s$. Therefore, the triangulation of $\ch_Q(A)$ decomposes into disjoint triangulations of the subregions $E^{(r)}(Q)$, $r=1,\ldots,k$.

Let now $\mathcal{P}(Q)$ be the induced tiling of $Z(V\setminus\{v_q\}_{q\in Q})$ via \eqref{eqn:indtil2}. Comparing \eqref{eqn:reglink3} with \eqref{eqn:indtil2} shows that the tiles $\{\Pi_{I,B}\in\bigsqcup_{r\geq 0}\mathcal{E}^{(r)}(Q)\}$ are in bijection with the tiles $\{\Pi_{\varnothing,B}\in\mathcal{P}(Q)\}\fd\mathcal{P}^{(0)}(Q)$. Therefore, by Proposition \ref{prop:triangulation}, the simplices $\{\ch(\{a_b\}_{b\in B}):\Pi_{\varnothing,B}\in\mathcal{P}^{(0)}(Q)\}$  form a triangulation of $\ch_Q(A)$, proving \ref{enum:linktiltot}.
\end{proof}

Based on these facts, we can replace Definition \ref{def:reglink} of the link region of $Q$, $\norm{Q}=k$, with
\begin{equation}
    \label{eqn:reglinkalt}
        \mathcal{R}(Q)\df \ch_Q(A)\setminus\left(\bigcup_{r=0}^{k-1}E^{(r)}(Q)\right),
\end{equation}
which is preferred from an algorithmic standpoint because it expresses $\mathcal{R}(Q)$ only in terms of the tiles $\Pi_{I,B}\in\mathcal{P}_r$ with $r<k$. Given that the simplex $\ch(\{a_b\}_{b\in B})$ is non-degenerate for any tile $\Pi_{I,B}$, Proposition \ref{prop:linktiling} implies that the region $\mathcal{R}(Q)\df E^{(k)}(Q)$ is empty if and only if its triangulation contains no simplices, i.e., if and only if $\mathcal{E}^{(k)}(Q)$ is empty. We have therefore the following corollary: 
\begin{corollary}
\label{cor:emptylink}
$\mathcal{R}(Q)$ is nonempty if and only if there is a tile $\Pi_{I,B}\in\mathcal{P}$ with $I=Q$. 
\end{corollary}

Proposition \ref{prop:linktiling} and Corollary \ref{cor:emptylink} together imply that any fine zonotopal tiling $\mathcal{P}$ of $Z(V)$, and therefore the associated family of spline spaces, can be obtained iteratively by triangulating the link region associated to each set $I$ for every tile $\Pi_{I,B}$ through some choice of triangulation, similarly to Liu and Snoeyink's algorithm in two dimensions. This statement can be made precise as follows:

\begin{theorem}
\label{thm:constructiongen}
For every fine zonotopal tiling $\mathcal{P}$ of $Z(V)$ there exists a choice of triangulations $\mathcal{T}_I$, one for every link region $\mathcal{R}(I)$ associated to each subset $\{I\subseteq [n]:\Pi_{I,B}\in\mathcal{P}\mbox{ for some }B\}$, such that $\mathcal{P}$ (and its associated spline spaces at all orders $0\leq k\leq n-d-1$) can be constructed as follows:
\begin{enumerate}[label=(\roman*)]
\item Let $\mathcal{I}^{(0)}=\{\varnothing\}$;
\item For every $0\leq k\leq n-d-1$ and for every $I\in\mathcal{I}^{(k)}$, let $\mathcal{R}(I)$ be the link region computed via \eqref{eqn:reglinkalt}, and let $\mathcal{T}_I$ be its triangulation. Denoting the simplex $\Delta_B\df\ch(\{a_b\}_{b\in B})$, the subset of tiles $\mathcal{P}^{(k)}\df\{\Pi_{I,B}\in\mathcal{P}:\norm{I}=k\}$ is given by
\begin{equation}
    \mathcal{P}^{(k)}=\{\Pi_{I,B}:I\in\mathcal{I}^{(k)},\,\Delta_B\in\mathcal{T}_I\};
\end{equation}\label{i:tri}
\item Let 
\begin{equation}
\label{eqn:intsetsp}
    \mathcal{I}^{(k+1)}=\{I\sqcup\{b\}:\Pi_{I,B}\in\mathcal{P}^{(k)},\, b\in B,\, \mathcal{R}(I\sqcup\{b\})\neq\varnothing\}
\end{equation}\label{i:next}
\item Repeat \ref{i:tri} and \ref{i:next} until $k=n-d-1$,  $\mathcal{I}^{(k+1)}=\varnothing$. Then $\mathcal{P}=\bigsqcup_{k=0}^{n-d-1}\mathcal{P}^{(k)}$.
\end{enumerate}
\end{theorem}
\begin{proof}
Item \ref{enum:linktilreg} of Proposition \ref{prop:linktiling} directly states that the tiles $\Pi_{I,B}\in\mathcal{P}^{(k)}$ (i.e., splines of degree $k$) are in bijection with the simplices $\ch(\{a_b\}_{b\in B})$ of a triangulation of the link region $\mathcal{R}(I)$. Furthermore, due to Corollary \ref{cor:emptylink}, all the tiles $\Pi_{I,B}\in\mathcal{P}^{(k)}$ are associated with a nonempty link region, which is always triangulable since Proposition \ref{prop:linktiling} exhibits one such triangulation. The only thing left to determine is the set $\{I:\Pi_{I,B}\in\mathcal{P}\}$.

Notice that $I\in\mathcal{I}^{(0)}$ implies $I=\varnothing$, and by \eqref{eqn:reglinkalt}, $\mathcal{R}(\varnothing)=\ch(A)$. Therefore, the tiles $\Pi_{\varnothing,B}$ (i.e., splines of degree $0$) are in bijection with the simplices of a triangulation of $\ch(A)$, in accordance with Proposition \ref{prop:triangulation}. 

Assume now that we have obtained all the tiles $\Pi_{I,B}\in\mathcal{P}^{(r)}$ for $r=0,\ldots,k$, and we want to determine the set $\mathcal{I}^{(k+1)}\df\{I:\Pi_{I,B}\in\mathcal{P}^{(k+1)}\}$. 

Let $Q\subset [n]$, $\norm{Q}=k+1$ be a set of indices such that $\mathcal{R}(Q)\neq\varnothing$, let $\{\Delta_f,f=1,\ldots,F\}$ be the $F$ boundary facets of $\mathcal{R}(Q)$, and for every $f=1,\ldots,F$, let $\Pi_{Q,B_f}$ and $b_f\in B_f$ be a tile in $\mathcal{P}^{(k+1)}$ such that $\Delta_f=\ch(\{a_i\}_{i\in B_f\setminus\{b_f\}})$. By Proposition \ref{prop:triangulation}, this tile is unique. Suppose that all the facets $\{\Pi_{Q,B_f\setminus\{b_f\}},f=1,\ldots,F\}$ lie on the boundary of $Z(V)$, let $\norm{\Delta_f}$ be the volume of $\Delta_f$ and let $N_f\in\mathbb{R}^d$ be its normalized normal vector. Without loss of generality, we can choose either all inward or all outward normal vectors so that $\sum_{f=1}^{F}\norm{\Delta_f}\left\langle N_f,a_{b_f}\right\rangle\leq 0$. Since $\mathcal{R}(Q)$ is a nonempty, bounded polyhedral region, we know that $\sum_{f=1}^{F}\norm{\Delta_f}N_f=0$, and we can therefore write the following linear dependency with positive coefficients $\norm{\Delta_f}_{f=1,\ldots,F}$, and $1$:
\begin{equation}
\label{eqn:posdep2}
\sum_{f=1}^{F}\norm{\Delta_f}\left(N_f,-\left\langle N_f, a_{b_f}\right\rangle\right)+(0,\sum_{f=1}^{F}\norm{\Delta_f}\left\langle N_f, a_{b_f}\right\rangle)=0.
\end{equation}
Fix a point $a_q$ with $q\in Q$. If, for all $f=1,\ldots,F$, $a_q$ were separated from $a_{b_f}$ by the hyperplane $\ch(\{a_i\}_{i\in B_f\setminus\{b_f\}})$, then we would have
\begin{empheq}{alignat=1}
\label{eqn:nosep}
\begin{aligned}
    (a_q,1)\cdot\left(N_f,-\left\langle N_f, a_{b_f}\right\rangle\right)=\left\langle N_f,a_q-a_{b_f}\right\rangle &< 0,\\ (a_q,1)\cdot(0,\sum_{f=1}^{F}\norm{\Delta_f}\left\langle N_f,a_{b_f}\right\rangle)=\sum_{f=1}^{F}\norm{\Delta_f}\left\langle N_f,a_{b_f}\right\rangle&\leq 0.
\end{aligned}
\end{empheq}
By Stiemke's Lemma \cite{stiemke1915positive}, \eqref{eqn:posdep2} and \eqref{eqn:nosep} cannot both be true. Therefore, there must be an index $f$ such that the facet $\Pi_{Q,B_f\setminus\{b_f\}}$ does not lie on the boundary of $Z(V)$. Observe also that $\Pi_{Q,B_f\setminus\{b_f\}}$ cannot be shared with another tile $\Pi_{I^\prime,B^\prime}\in\mathcal{P}^{(k+1)}$, since otherwise $I^\prime=Q$ and $\Delta_f$ would not be a boundary facet of $\mathcal{R}(Q)$. Therefore, by Proposition \ref{prop:faces}, there must be a tile $\Pi_{I,B}\in\mathcal{P}^{(k)}$ with $B_f\setminus\{b_f\}=B\setminus\{b\}=B_f\cap B$ and $Q=I\sqcup\{b\}$ for some $b\in B$. We conclude that
\begin{equation}
    \mathcal{I}^{(k+1)}\subseteq\{I\sqcup\{b\}:\Pi_{I,B}\in\mathcal{P}^{(k)},\,b\in B\}.
\end{equation}
After filtering out the sets $\{I\sqcup\{b\}:\mathcal{R}(I\sqcup\{b\})=\varnothing\}$, we are left exactly with \eqref{eqn:intsetsp}. 

Finally, when $\norm{Q}=n-d$, the set $\ch_Q(A)$ only contains $d$ points, and therefore the link region $\mathcal{R}(Q)$ has an empty interior. Therefore, $\mathcal{I}^{(n-d)}=\varnothing$, and the process stops. 
\end{proof}
This theorem states essentially that any fine zonotopal tiling of $Z(V)$ can be built using  a version of Liu and Snoeyink's algorithm, provided that we know in advance which triangulation needs to be applied to each subset $\{I:\Pi_{I,B}\in\mathcal{P}\}$. In other words, it proves that their algorithm is a universal way of constructing fine zonotopal tiling over $Z(V)$ and their associated spline spaces. However, this result stops short of providing a fully-formed construction algorithm, as it does not guarantee that any given choice of triangulations leads to a valid construction, only that such a choice exists. In the next section, we show that {\em regular} fine zonotopal tilings can be obtained by choosing a weighted Delaunay triangulation at each step, providing a sufficient condition on the triangulations that guarantees the convergence of the construction process.

Finally, we give a couple of interesting results regarding the combinatorial structure of spline spaces built by Theorem \ref{thm:constructiongen}. First, as a direct consequence of Theorem \ref{thm:zonsub}, we obtain the following simple characterization of the total number of spline functions:
\begin{corollary}
The total number of spline functions built by the process described in Theorem \ref{thm:constructiongen} on a point set $A$ with $\norm{A}=n$, summed over all orders $k=0,\ldots,n-d-1$, is always equal to the number of affinely independent subsets of $A$.
\end{corollary}

Next, we provide a characterization of the set of simplices
\begin{equation}
	\mathcal{T}^{(k)}\df\{\ch(\{a_b\}_{b\in B}):\Pi_{I,B}\in\mathcal{P}^{(k)}\}
\end{equation}
The intersection of these simplices defines the zones where all the spline functions are pure polynomials, and their boundaries define the zones of reduced regularity of spline functions, i.e., knots in $d=1$, knot lines in $d=2$ and more generally knot hypersurfaces in $d>2$.
\begin{proposition}
\label{prop:cover}
For all $0\leq k\leq n-d-1$, the simplices in $\mathcal{T}^{(k)}$ cover $\binom{k+d}{d}$ times the set $\ch_k(A)$.
\end{proposition}
\begin{proof}
By induction over $k$. The simplices in $\mathcal{T}^{(0)}$ form a triangulation of $\ch(A)$ by Proposition \ref{prop:triangulation}, and therefore cover it exactly once. Assume now that the proposition is true for every $r<k$. By Property \ref{enum:linktiltot} of Proposition \ref{prop:linktiling}, for any subset $Q\subset [n]$ with $\norm{Q}=k$, the simplices $\{\ch(\{a_b\}_{b\in B}):\,\Pi_{I,B}\in\mathcal{E}^{(r)}(Q),\, r\leq k\}$ triangulate $\ch_Q(A)$, i.e.,
\begin{equation}
\label{eqn:sumq}
	\sum_{r=0}^k\sum_{\Pi_{I,B}\in\mathcal{E}^{(r)}(Q)}\indicf{B}=\indicf{\ch_Q(A)},
\end{equation}
where $\indicf{\ch_Q(A)}:\mathbb{R}^d\mapsto\mathbb{R}$ is the indicator function of the set $\ch_Q(A)\subset\mathbb{R}^d$ and $\indicf{B}$ is the indicator function of $\ch(\{a_b\}_{b\in B})$. We sum this expression over all subsets $Q\subset [n]$, $\norm{Q}=k$. Each tile $\Pi_{I,B}\in\mathcal{P}^{(r)}$ appears in the sum whenever $I\sqcup J=Q$ for some subset $J\subset [n]$, $\norm{J}=k-r$ with $J\cap B=\varnothing$. Therefore, the occurrences of a tile of $\mathcal{P}^{(r)}$ in the sum correspond to the possible choices of $\norm{Q\setminus I}=k-r$ indices among the $\norm{\overline{I\sqcup B}}=n-r-d-1$ which are available. We obtain
\begin{equation}
\label{eqn:sumq2}
\sum_{\Pi_{I,B}\in\mathcal{P}^{(k)}}\indicf{B}+\sum_{r=0}^{k-1}\binom{n-r-d-1}{k-r}\sum_{\Pi_{I,B}\in\mathcal{P}^{(r)}}\indicf{B}=\sum_{Q\subset [n],\norm{Q}=k}\indicf{\ch_Q(A)}.
\end{equation}
By induction, the simplices derived from the tiles in $\mathcal{P}^{(r)}$ cover the region $\ch_r(A)\supseteq\ch_k(A)$ exactly $\binom{r+d}{d}$ times, and the sum on the right covers $\ch_k(A)$ exactly $\binom{n}{k}$ times. Using multiset notation and the Vandermonde identity, we can derive
\begin{align}
\label{eqn:vandermonde}
    \sum_{r=0}^k\binom{n-r-d-1}{k-r}\binom{r+d}{r}&=\sum_{r=0}^k\multichoose{n-k-d}{k-r}\multichoose{d+1}{r}\\
    &=\multichoose{n-k+1}{k}=\binom{n}{k}.
\end{align}
Separating the term with $r=k$ in the first sum in \eqref{eqn:vandermonde}, we conclude that the first term in \eqref{eqn:sumq2}, i.e. the set of all simplices in $\mathcal{T}^{(k)}$, must cover the region $\ch_k(A)$ exactly $\binom{k+d}{d}$ times.
\end{proof}

Notice that in general it is not possible to extract from the set $\mathcal{T}^{(k)}$ a collection of $\binom{k+d}{d}$ independent triangulations, as these simplices form in general a {\em branched cover} of $\ch(A)$. In practice, $\mathcal{T}^{(k)}$ forms a complex web of overlapping simplices that contains many complex intersections, see e.g. Figure \ref{fig:reglink}. 
\begin{figure}[htbp]
    \centering
    \subfloat{\includegraphics[width=0.25\textwidth]{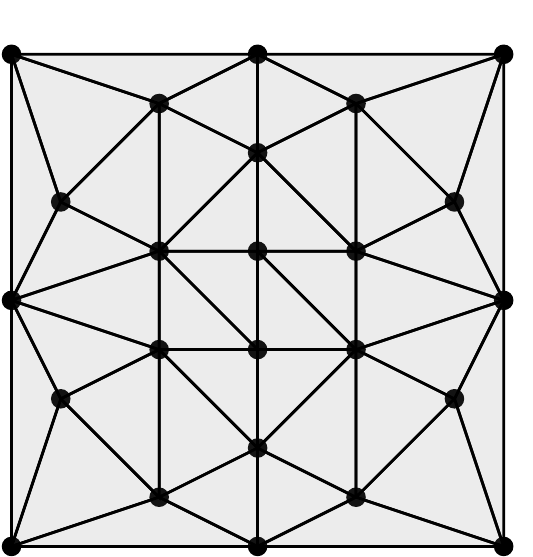}}%\hspace{0.05\textwidth}%
    \subfloat{\includegraphics[width=0.25\textwidth]{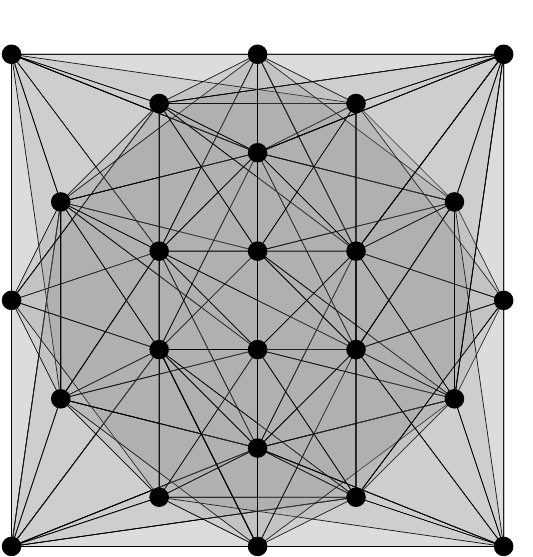}}
    \subfloat{\includegraphics[width=0.25\textwidth]{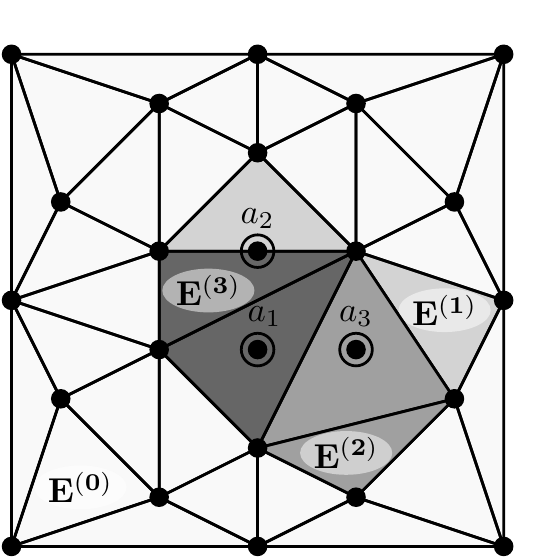}}%\hspace{0.05\textwidth}%
    \subfloat{\includegraphics[width=0.25\textwidth]{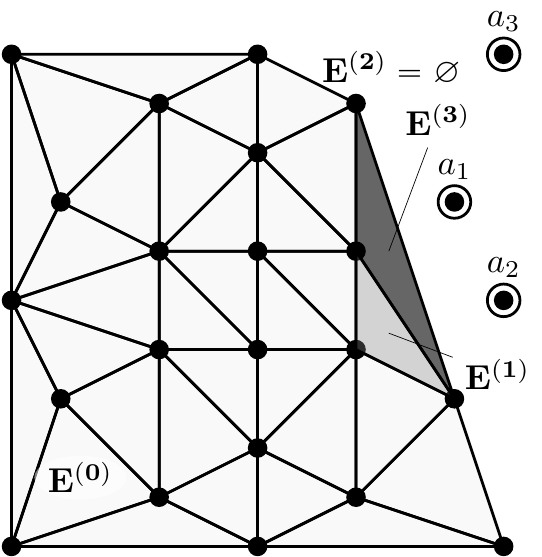}}
	\caption{\small For a point configuration $A\subset\mathbb{R}^2$ with collinear points, the sets $\mathcal{T}^{(k)}$ for $k=0$ and $k=2$, with the shading indicating the number of simplices covering each point, and the regions $E^{(r)}(Q)$ of \eqref{eqn:reglink2} for two possible choices of $Q\df(a_1,a_2,a_3)$.}
    \label{fig:reglink}
\end{figure}

\section{Spline spaces from regular fine zonotopal tilings}
\label{sec:del}
We specialize the results of the previous section to spline spaces derived from {\em regular} fine zonotopal tilings. Given a polytope $P\subset\mathbb{R}^{d+2}$, we define its {\em upper convex hull} as the set of faces of $P$ whose outward normal vector has a positive $(d+1)$-th component.
\begin{definition}
A zonotopal tiling $\mathcal{P}$ of $Z(V)\subset\mathbb{R}^{d+1}$ is {\em regular} if its tiles are precisely the projections along the $(d+1)$-th coordinate of the faces in the upper convex hull of another zonotope $\tilde{Z}\subset\mathbb{R}^{d+2}$.
\end{definition}
We show that this special case corresponds exactly to simplex splines associated to weighted Delaunay configurations. The special properties of these tilings then allow us to derive a set of practical algorithms for the construction of the spline spaces and the determination and evaluation of all spline functions that are supported on a given point $x\in\mathbb{R}^d$.

\subsection{Delaunay triangulations and regular zonotopal tilings}
\label{sec:del:const}
Let $h:A\mapsto\mathbb{R}$ be a {\em height function} over $A$. Let $\mathcal{T}$ be a set of simplices that triangulate $\ch(A)$ with vertices in $A$. For every subset $B\subseteq[n]$, $\norm{B}=d+1$ such that there is a simplex $\Delta\df\ch(\{a_b\}_{b\in B})\in\mathcal{T}$, let us order $B$ such that $\det((a_b,1)_{b\in B})>0$. If, for every $i\in A\setminus B$,
\begin{equation}
\label{eqn:delcond}
    \det((a_b,h(a_b),1)_{b\in B},(a_i,h(a_i),1))<0,
\end{equation}
then the triangulation $\mathcal{T}$ is called a {\em weighted Delaunay triangulation} with height function $h$. If the points of $A$ are in general position, plugging $h(a)=\dnorm{a}^2$ in \eqref{eqn:delcond} yields the usual Delaunay triangulation, see e.g. \cite{guibas1985primitives}.

In order for the Delaunay triangulation to exist and be unique, a bit of care is needed when choosing the height function $h$. 
\begin{definition}
\label{def:genericw}
A height function $h$ is {\em generic} if, given the lifted point cloud
\begin{equation}
\tilde{A}\df\{(a,h(a)),\,a\in A\}\subset\mathbb{R}^{d+1},
\end{equation}
the only affinely dependent subsets of $d+2$ points in $\tilde{A}$ lie on a vertical plane, i.e., a plane whose normal $N\in\mathbb{R}^{d+1}$ satisfies $N_{d+1}=0$.
\end{definition}
Notice that affinely dependent subsets are indeed allowed on vertical planes, and thus the points in $A$ can be repeated or affinely dependent. If $h$ is generic, then the determinant in \eqref{eqn:delcond} is always nonzero, and the weighted Delaunay triangulation is unique. Hereafter, we will only consider generic height functions. We can now use \eqref{eqn:delcond} to specialize Theorem \ref{thm:constructiongen} to weighted Delaunay triangulations.
\begin{theorem}
\label{thm:constructiondel}
Let $h$ be a generic height function on $A$, and for every set $Q\subseteq [n]$ let $\mathcal{T}_Q(h)$ be the weighted Delaunay triangulation of $\ch_Q(A)$ with height function $h$. Then the procedure outlined in Theorem \ref{thm:constructiongen} with the choice $\mathcal{T}_I=\mathcal{T}_I(h)$ always produces a regular fine zonotopal tiling $\mathcal{P}(h)$.
\end{theorem}
\begin{proof}
It is easy to prove using the lifting property \eqref{eqn:delcond}. See also \cite{edelsbrunner2018,edelsbrunner2019poisson} and especially \cite{santos1996delaunay} for similar constructions and an interesting generalization.

Let $\tilde{A}=\{\tilde{a}_i\df(a_i,h(a_i)),\,i=1,\ldots,n\}\subset\mathbb{R}^{d+1}$ be the point cloud lifted by $h$, $\tilde{V}\df\{(a_i,h(a_i),1):i=1,\ldots,n\}$ be the associated vector configuration and $Z(\tilde{V})$ be the zonotope built on $\tilde{V}$. Denoting by $\pi:\mathbb{R}^{d+2}\mapsto\mathbb{R}^{d+1}$ the projection that removes the $(d+1)$-th coordinate, it is easy to check that $\pi(Z(\tilde{V}))=Z(V)$. We define $\mathcal{P}(h)$ as the regular zonotopal tiling
\begin{equation}
\label{eqn:regzt}
    \mathcal{P}(h)\df\{\pi(\tilde{\Pi}_{I,B}):\tilde{\Pi}_{I,B}\mbox{ is in the upper convex hull of }Z(\tilde{V})\}.
\end{equation}
The fact that \eqref{eqn:regzt} is indeed a regular zonotopal tiling of $Z(V)$ was proven e.g. in \cite[Lemma~2.2]{billera1992fiber}. Since $\tilde{\Pi}_{I,B}$ is a boundary facet of $Z(\tilde{V})$, we can follow the same reasoning as in the proof of item \ref{item:faces2} of Proposition \ref{prop:faces}. After selecting the face normal $N_B$ of $\tilde{\Pi}_{I,B}$ with $(N_B)_{d+1}>0$, given that $h$ is generic and the face is not vertical, we conclude that the determinant
\begin{equation}
\label{eqn:detdel}
    \det((a_b,h(a_b),1)_{b\in B},(a_i,h(a_i),1))
\end{equation}
is positive for all $i\in I$ and negative for all $i\in\overline{I\sqcup B}$, while the condition $(N_B)_{d+1}>0$ translates to $\det((a_b,1)_{b\in B})>0$. Since only the points $\{a_i\}_{i\in\overline{I\sqcup B}}$ appear in the link region $\mathcal{R}(I)$, the weighted Delaunay condition \eqref{eqn:delcond} is satisfied for all the points in $\mathcal{R}(I)$.
\end{proof}

Theorems \ref{thm:constructiongen} and  \ref{thm:constructiondel} together give a practical construction algorithm for all regular fine zonotopal tilings of $Z(V)$, and therefore for their associated spline spaces. Restricting the construction to the the special case $d=2$ and to points in generic position, this process reduces to a version of Liu and Snoeyink's construction algorithm \cite{liu2007quadratic,liu2008computations,schmitt2019bivariate}. 

\subsection{Splines supported on a point}
\label{sec:del:supp}
In this subsection we show that, in the case of spline spaces associated to regular fine zonotopal tilings, there exists an efficient process to determine all the spline functions up to a given degree $k\geq 0$ that are supported on a given point $x\in\mathbb{R}^d$. This is equivalent, by \eqref{eqn:splinedef}, to finding all the tiles $\Pi_{I,B}\in\mathcal{P}(h)$ such that $x\in\ch(\{a_i\}_{i\in I\sqcup B})$. In this case, by extension, we say that the tile $\Pi_{I,B}$ is supported on $x$.

For spline functions of degree $0$, the task is particularly simple. In fact, since the simplices $\mathcal{T}^{(0)}$ triangulate $\ch(A)$ (Proposition \ref{prop:triangulation}), whenever $x\in\ch(A)$ there is one and only one tile $\Pi_{\varnothing,Z}$ supported on $x$. Computationally, $\Pi_{\varnothing,Z}$ can be found efficiently via a point location query on a triangulation, for which many efficient algorithms exist, see e.g. \cite{guttman1984r,beckmann1990r}. We prove in the remainder of this section that all the other tiles $\Pi_{I,B}$ (and hence spline functions) supported on $x$ can be found from $\Pi_{\varnothing,Z}$ using a suitable orientation, induced by $x$, of the adjacency graph $\mathcal{G}$ of $\mathcal{P}(h)$, i.e., the simple, connected graph having the tiles of $\mathcal{P}(h)$ as vertices and their connecting internal facets as edges.

We assume hereafter that the test point $x\in\mathbb{R}^d$ is {\em generic}, i.e., it satisfies the following condition:
\begin{equation}
\label{eqn:xgeneral}
	x\not\in\aff(\{a_c\}_{c\in C})\mbox{ for all internal facets }\Pi_{J,C}\mbox{ of }\mathcal{P}.
\end{equation}
This excludes from the possible values of $x$ a zero-measure subset of $\mathbb{R}^d$, and as a consequence, all the following results must be understood to hold almost everywhere. This restriction can be easily lifted using some well-known techniques such as symbolic perturbation. We can define an orientation $o_x$, depending on $x$, on the adjacency graph $\mathcal{G}$ of $\mathcal{P}$ as follows. Let $\Pi_{J,C}$ be a facet shared by two tiles $\Pi_{I,B}$ and $\Pi_{I^\prime,B^\prime}$, with normal vector $N_C\in\mathbb{R}^{d+1}$. Then we define the orientation of the corresponding edge in $\mathcal{G}$ as $\Pi_{I,B}\rightarrow\Pi_{I^\prime,B^\prime}$ if and only if
\begin{equation}
\label{eqn:orient}
    \sign\left(\scal{N_C}{(x,1)}\right)=\sign\left(\scal{N_C}{z^\prime-z}\right)
\end{equation}
for any $z^\prime\in\Pi_{I^\prime,B^\prime}$, $z\in\Pi_{I,B}$. In other words, we pick the direction of $N_C$ that leads to a positive scalar product with $(x,1)$, and we use it to orient the corresponding edge.

The orientation $o_x$ defined by \eqref{eqn:orient} yields a directed graph $(\mathcal{G},o_x)$. In the case of regular tilings, this graph is acyclic.
\begin{lemma}
\label{lem:acyclic}
Let $\mathcal{P}(h)$ be a regular fine zonotopal tiling of $Z(V)$ with generic height function $h$. Then the directed graph $(\mathcal{G},o_x)$ is acyclic for every generic $x\in\mathbb{R}^d$. The same is true for any fine zonotopal tiling $\mathcal{P}$ of $Z(V)$, regular or not, when $d=1$.
\end{lemma}
\begin{proof}
Let $\Pi_i\df\Pi_{I_i,B_i}$, $i=1,\ldots,r$ be a family of $r$ tiles of $\mathcal{P}(h)$ and let $F_i\df\Pi_{J_i,C_i}$, $i=1,\ldots, r$ be a family of facets such that $F_i$ is shared between the tiles $\Pi_i$ and $\Pi_{i+1}$. Let us assume that the tiles form a cycle in $\mathcal{G}$, i.e., $\Pi_{r+1}=\Pi_1$. For each $1\leq i\leq r$, let $N_i\df N_{C_i}$ be a vector normal to the $i$-th facet and pointing from the tile $\Pi_i$ to the tile $\Pi_{i+1}$.

Since $\mathcal{P}(h)$ is regular, by Theorem \ref{thm:constructiondel}, for each tile $\Pi_i$ there is a vector $y_i\in\mathbb{R}^{d+2}$ with $(y_i)_{d+1}>0$ such that $\scal{y_i}{(a_s,h(a_s),1)}$ is positive if $s\in I_i$, zero if $s\in B_i$ and negative if $s\in\overline{I_i\sqcup B_i}$ . Define the point $g_i\in\mathbb{R}^{d+1}$ component-wise as 
\begin{equation}
\label{eqn:centers}
    (g_i)_j\df\frac{(y_i)_j}{(y_i)_{d+1}},\,j=1,\ldots,d,\;(g_i)_{d+1}\df\frac{(y_i)_{d+2}}{(y_i)_{d+1}},
\end{equation}
which is possible since $(y_i)_{d+1}>0$. For all $b\in B_i$, $\scal{y_i}{(a_b,h(a_b),1)}=0$ implies
\begin{equation}
\label{eqn:centerscal}
\scal{g_i}{v_b}=-h(a_b),
\end{equation}
and as a consequence, for all $c\in B_i\cap B_{i+1}=C_i$,
\begin{equation}
\label{eqn:centerscal2}
    \scal{g_{i+1}-g_i}{v_c}=0,
\end{equation}
i.e., the vector $(g_{i+1}-g_i)$ is parallel to $N_i$. Let now $z_i\in\Pi_i$ be the point
\begin{equation}
\label{eqn:tilebary}
    z_i\df\sum_{j\in I_i}v_j+\frac{1}{2}\sum_{b\in B_i}v_b,
\end{equation}
and let $b\in B_i$, $b^\prime\in B_{i+1}$ be the two indices such that $B_i\setminus\{b\}=B_{i+1}\setminus\{b^\prime\}$. Let $\sigma_1=+1$ or $-1$ if $b\in I^\prime$ or $b\not\in I^\prime$, respectively, and similarly $\sigma_2=+1$ or $-1$ if $b^\prime\in I$ or $b^\prime\not\in I$ respectively. Using \eqref{eqn:detdel}, \eqref{eqn:centers} and \eqref{eqn:tilebary}, it is easy to check that $\sign(\scal{g_i}{v_{b^\prime}}+h(a_{b^\prime}))=\sigma_2$, $\sign(\scal{g_{i+1}}{v_b}+h(a_b))=\sigma_1$ and $z_{i+1}-z_i=\sigma_1v_b-\sigma_2 v_{b^\prime}$. Therefore, according to \eqref{eqn:centerscal} and \eqref{eqn:centerscal2},
\begin{align}
\sign(\scal{g_{i+1}\!-\!g_i}{z_{i+1}\!-\!z_i})\!&=\sign(\scal{g_{i+1}\!-\!g_i}{\sigma_1 v_b\!-\!\sigma_2 v_{b^\prime}})\\
\!&=\sign(\sigma_1\scal{g_{i+1}}{v_b}\!+\!\sigma_2 h(a_{b^\prime})\!+\!\sigma_1 h(a_b)\!+\!\sigma_2\scal{g_i}{v_{b^\prime}})\\
\!&=\sigma_1^2+\sigma_2^2>0.
\end{align}
In other words, $(g_{i+1}-g_i)$ always points in the same direction as $N_i$, and thus $g_{i+1}-g_i=\mu_i N_i$ for some $\mu_i>0$. We can therefore write:
\begin{equation}
\label{eqn:combzero}
    0=\sum_{i=1}^r (g_{i+1}-g_i)=\sum_{i=1}^r\mu_i N_i\mbox{ with }\mu_1,\ldots,\mu_r>0.
\end{equation}
Taking the scalar product of \eqref{eqn:combzero} with $(x,1)$, $x\in\mathbb{R}^d$ shows that, for at least one facet $F_i$, we must have $\scal{N_i}{(x,1)}<0$ and therefore
\begin{equation}
    \sign\left(\scal{N_i}{(x,1)}\right)\neq \sign\left(\scal{N_i}{z_{i+1}-z_i}\right),
\end{equation}
i.e., \eqref{eqn:orient} fails. In other words, this orientation cannot be induced by any generic point $x\in\mathbb{R}^d$. All orientations $(\mathcal{G},o_x)$ are therefore acyclic.

In the one-dimensional case, we can obtain the positive linear combination of normals \eqref{eqn:combzero} without assuming the existence of the vectors $y_i$. We only give a sketch of the proof. First, there is at least one tile $\Pi_i$ such that $F_i\neq F_{i+1}$, else the tiles cannot form a loop. Furthermore, since each tile is convex, each angle $N_i\angle N_{i+1}$ can only be strictly less than $\pi$, but the total angle along the cycle must be equal to $2k\pi$, $k\in\mathbb{Z}\setminus\{0\}$. These conditions imply that there is a closed path in $\mathbb{R}^2$ whose $j$-th displacement vector is directed along $N_j$. Defining $g_i$ as the $i$-th vertex of the path then yields \eqref{eqn:combzero}.

\end{proof}
\begin{remark}
The construction used in the proof of Lemma \ref{lem:acyclic} is similar to the affinization of central hyperplane arrangements, see e.g. \cite[Chapter~7]{beck2018combinatorial}.
\end{remark}

As a directed acyclic graph, $(\mathcal{G},o_x)$ can be topologically sorted, and the (only) tile $\Pi_{\varnothing,Z}$ supported on $x$ can be used as the root of an oriented path that follows the topological sorting. We prove now that the other tiles $\Pi_{I^\prime,B^\prime}$ supported on $x$ are all reachable from $\Pi_{\varnothing,Z}$ using such a path. First, we need a small lemma in convex theory, very similar (although not equivalent) to Carathéodory's theorem.
\begin{lemma}
\label{lem:sephyper}
Let $A=(a_1,\ldots,a_n)$ be a configuration of $n>d+1$ points in $\mathbb{R}^d$, and let $B\subset [n]$ be a set of $\norm{B}=d+1$ indices such that the points $(a_i)_{i\in B}$ are affinely independent. Then, for every $x\in\ch(A)$ there exists an index $b\in B$ such that $a_b$ and $x$ are on the same closed halfspace of $\aff(\{a_i\}_{i\in B\setminus\{b\}})$ and $x\in\ch(\{a_i\}_{i\in [n]\setminus \{b\}})$.
\end{lemma}
\begin{proof}
First, assume that $x\in\ch(\{a_i\}_{i\in B})$. In this case, for all $b\in B$, $x$ is on the same closed halfspace of $\aff(\{a_i\}_{i\in B\setminus\{b\}})$ as $a_b$. We can then pick any index $c\in [n] \setminus B$, and the (possibly degenerate) simplices $\ch(\{a_i\}_{i\in B\setminus\{b\}\sqcup\{c\}})$ for all $b\in B$ cover $\ch(\{a_i\}_{i\in B})$. Thus, for at least one index $b\in B$, $x\in\ch(\{a_i\}_{i\in B\setminus\{b\}\sqcup\{c\}})$, satisfying the lemma.

Assume now that $x\not\in\ch(\{a_i\}_{i\in B})$. Then $x\in\ch(A)$ if and only if
\begin{equation}
 \label{eqn:positivecomb}
 x=\sum_{i=1}^n\mu_i a_i
 \end{equation}
for some real numbers $\mu_a=i$ satisfying $\mu_i\geq 0$ and $\sum_{i=1}^n\mu_i=1$. Since the points indexed by $B$ are affinely independent, we can also express $x=\sum_{b\in B}\lambda_b a_b$, with $\sum_{b\in B}\lambda_b=1$. We extend this to a linear combination $x=\sum_{i=1}^n\lambda_i a_i$ by defining $\lambda_i\df 0$ for $i\not\in B$.
We have
\begin{equation}
    \sum_{i=1}^n\mu_i=1=\sum_{i=1}^n\lambda_i,
\end{equation}
and therefore $\sum_{i=1}^n(\mu_i-\lambda_i)=0$. The expression $\mu_i-\lambda_i$ cannot be identically zero for all $i\in [n]$, since otherwise $x\in\ch(\{a_j\}_{j\in B})$, which we have excluded. Thus, there must be at least one $b\in B$ with $\lambda_b>\mu_b\geq 0$. If we pick an index $c\in B$ such that \begin{equation}
    c\in\argmin_{b\in B}\left\{\alpha_b\df\frac{\mu_b}{\lambda_b-\mu_b}:\lambda_b>\mu_b\right\},
\end{equation}
we can write the nonnegative linear combination
\begin{equation}
\label{eqn:positivecomb3}
 \sum_{i=1}^n\left[\mu_i-(\lambda_i-\mu_i)\alpha_c\right]a_i=x,
\end{equation}
where clearly $\mu_i-(\lambda_i-\mu_i)\alpha_c\geq 0$ and $\mu_c-(\lambda_c-\mu_c)\alpha_c=0$. Thus, the point $a_c$ satisfies the lemma, since $\lambda_c>\mu_c\geq 0$ implies that $a_c$ and $x$ are on the same open halfspace of $\aff(\{a_i\}_{i\in B\setminus\{c\}})$, and $x$ can be expressed as the convex combination \eqref{eqn:positivecomb3} with the point $a_c$ having a zero coefficient.
\end{proof}

We can now prove that there is always a directed path in $(\mathcal{G},o_x)$ from $\Pi_{\varnothing,Z}$ to any tile $\Pi_{I^\prime,B^\prime}$ supported on $x$.
\begin{proposition}
\label{prop:support}
Let $\mathcal{P}(h)$ be a regular fine zonotopal tiling of $Z(V)$ with generic height function $h$, let $x\in\ch(A)$ be a generic point, and let $\Pi_{\varnothing,Z}$ be the only tile in $\mathcal{P}^{(0)}(h)$ supported on $x$. Then for every tile $\Pi_{I^\prime,B^\prime}\in\mathcal{P}(h)$ supported on $x$, there is a directed path in $(\mathcal{G},o_x)$ from $\Pi_{\varnothing,Z}$ to $\Pi_{I^\prime,B^\prime}$ with every tile $\Pi_{I,B}$ in the path satisfying $\norm{I}\leq\norm{I^\prime}$.
\end{proposition}
\begin{proof}
If $I^\prime=\varnothing$, then necessarily $\Pi_{I^\prime,B^\prime}=\Pi_{\varnothing,Z}$, and we are done. Else, we complete the proof by finding another tile $\Pi_{I,B}$ and an oriented edge $\Pi_{I,B}\rightarrow\Pi_{I^\prime,B^\prime}$ in $(\mathcal{G},o_x)$ such that $\Pi_{I,B}$ is supported on $x$ and $I\subseteq I^\prime$. The same reasoning can then be applied to $\Pi_{I,B}$ and again repeatedly, yielding an oriented path of tiles supported on $x$ and with non-increasing $\norm{I}$. Since the graph is acyclic (Lemma \ref{lem:acyclic}) and the number of tiles is finite, the process must eventually end with $\Pi_{I,B}=\Pi_{\varnothing,Z}$ as the root of the path.

According to Lemma \ref{lem:sephyper}, and since $x$ is generic, there exists an index $b^\prime\in B^\prime$ such that
\begin{equation}
\label{eqn:ch}
    x\in\ch(\{a_i\}_{i\in I^\prime\sqcup B^\prime\setminus\{b^\prime\}})\mbox{ and }a_{b^\prime},\,x\mbox{ are on the same side of }H_{b^\prime},
\end{equation}
where $H_{b^\prime}\df\aff(\{a_i\}_{i\in B^\prime\setminus\{b^\prime\}})$. Necessarily, this means that there is an index $j\in I^\prime$ such that $a_j$ is on the same side of $H_{b^\prime}$ as $a_{b^\prime}$, otherwise $H_{b^\prime}$ would separate $x$ from the convex hull $\ch(\{a_i\}_{i\in I^\prime\sqcup B^\prime\setminus\{b^\prime\}})$ and \eqref{eqn:ch} would be false. Proposition \ref{prop:faces} then guarantees that there is a tile $\Pi_{I,B}$, connected to $\Pi_{I^\prime,B^\prime}$ with an edge in $\mathcal{G}$, such that $B\setminus\{b\}=B^\prime\setminus\{b^\prime\}$ for some $b\in B$ and either $I^\prime=I$ or $I^\prime=I\sqcup\{b\}$. The point $a_b$ is on the opposite side of $H_{b^\prime}$ as $a_{b^\prime}$ and $x$ in the first case, and on the same side in the second case. It is easy to check, using \eqref{eqn:orient} and taking the representative points $z\in\Pi_{I,B}$ and $z^\prime\in\Pi_{I^\prime,B^\prime}$ defined as in \eqref{eqn:tilebary}, that in both cases the edge associated to the tile $\Pi_{J,C}$ with $J=I^\prime$, $C=B\cap B^\prime$ is oriented from $\Pi_{I,B}$ to $\Pi_{I^\prime,B^\prime}$. Furthermore, in both cases, $B\sqcup I\supseteq I^\prime\sqcup B^\prime\setminus\{b^\prime\}$, implying that $\Pi_{I,B}$ is supported on $x$, and $I\subseteq I^\prime$. This completes the proof.
\end{proof}

Proposition \ref{prop:support} is important because it shows that every tile $\Pi_{I,B}$ of order $k$ can be connected to $\Pi_{\varnothing,Z}$ in $(\mathcal{G},o_x)$ using only tiles of order $k$ or less (see e.g. Figure \ref{fig:cocircgraphor}). In practical applications, this implies that all the spline functions of degree $k$ supported on any given point can be found efficiently using only the knowledge of spline functions of degree $r\leq k$. Therefore, when constructing a spline space using the process delineated in Theorems~\ref{thm:constructiongen} and  \ref{thm:constructiondel}, the iterations can be safely stopped at the desired degree, without any need to access higher-degree functions.

Furthermore, Theorem \ref{prop:support} suggests a simple and efficient algorithm to find all the spline functions supported on a point $x$. The first step, which requires finding the spline of degree $k=0$ having $x$ in its support, can be efficiently implemented via any search tree constructed on the simplices in $\mathcal{T}^{(0)}$ \cite{guttman1984r,beckmann1990r}. Such trees typically have a $O\left(n\log(n)\right)$ construction complexity and a $O\left(\log(n)\right)$ query complexity, $n$ being the number of degree-zero splines. After this first step, the complexity is simply linear in the number of spline functions (of all degrees $r\leq k$) which are nonzero on $x$, and does not depend on the total number of functions in the spline space.

Notice however that there is still a need to check explicitly if every visited spline function is actually supported on $x$, albeit only for a limited number of functions.

We show an example of the directed graph $(\mathcal{G},o_x)$ in Figure \ref{fig:cocircgraphor}.

\begin{figure}[htbp]
    \centering
    \subfloat{\includegraphics[width=0.40\linewidth]{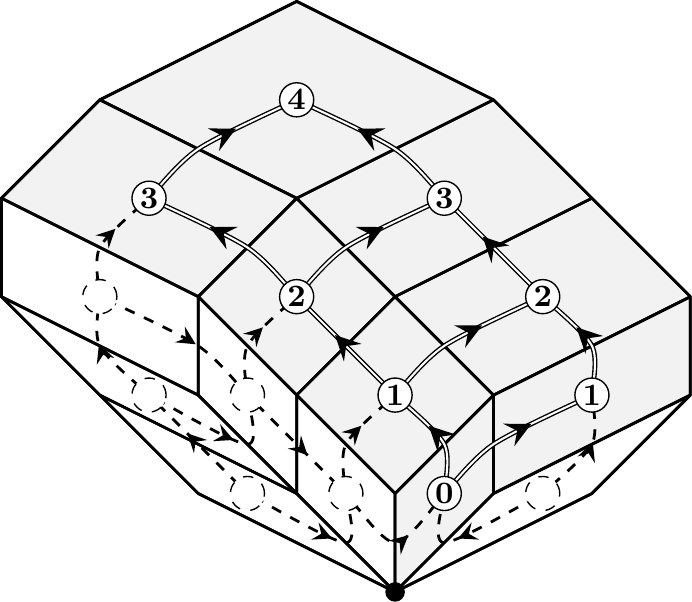}}\hspace{0.05\textwidth}%
    \subfloat{\includegraphics[width=0.40\linewidth]{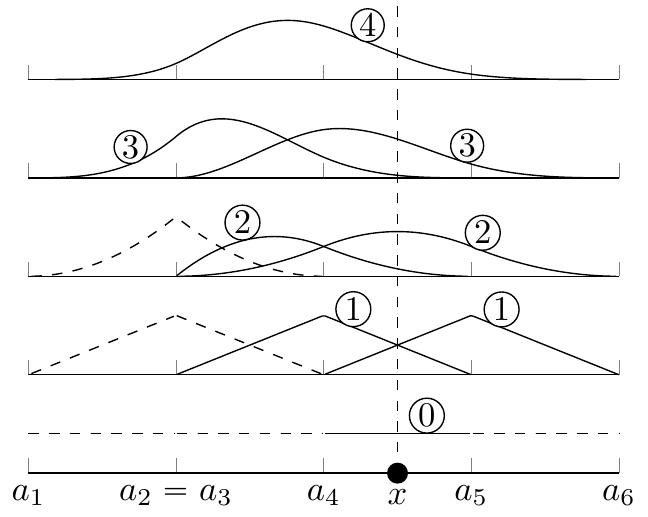}}\newline
    \subfloat{\includegraphics[width=0.40\linewidth]{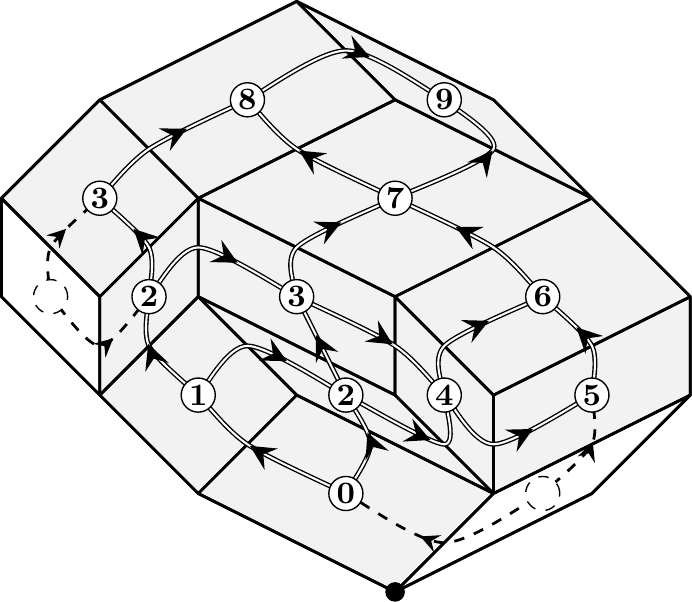}}\hspace{0.05\textwidth}l%
    \subfloat{\includegraphics[width=0.40\linewidth]{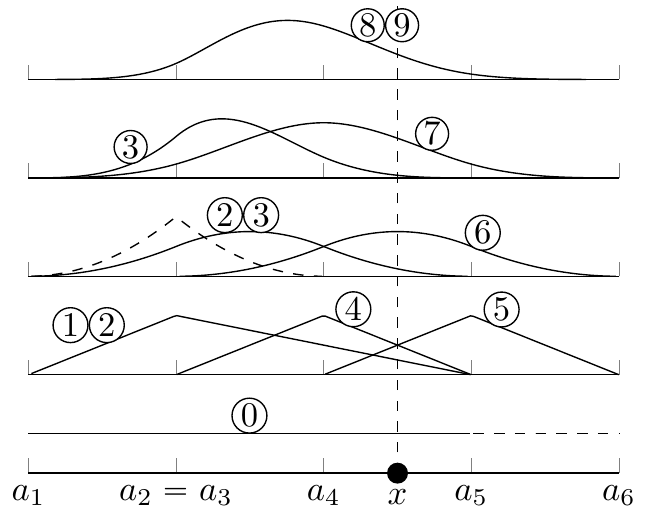}}
    \caption{{\em Left}: oriented adjacency graph $(\mathcal{G},o_x)$ for the tilings of Figure \ref{fig:zonobase}, with the orientation induced by a point $x\in(a_4,a_5)$. The subgraph determined by the tiles supported on $x$ is drawn with solid lines, and the tiles are numbered according to their position in a topological sorting of $(\mathcal{G},o_x)$, starting with $0$ for the tile $\Pi_{\varnothing,Z}$. {\em Right}: corresponding spline functions supported on $x$.}
    % how to obtain this basis: points $A=\{-2,-1,-1,0,1,3\}$, $h=\{2, 1, 1.6, 0.01, -1.5, 0\}$
    \label{fig:cocircgraphor}
\end{figure}

\subsection{Spline evaluation}
\label{sec:del:eval}
Once all the spline functions supported on a given point $x$ have been determined, one might be tempted to use the oriented graph $(\mathcal{G},o_x)$ and its topological sorting to compute the value of all the spline functions on $x$. 

Imagine that we want to compute, for some tile $\Pi_{I,B}$ supported on $x$, the value of $\overline{M}_b \df\M{x}{(a_i)_{i\in I\sqcup B\setminus\{b\}}}$ for all $b\in B$, which  can  in turn be used to compute the value of the spline itself $\overline{M}\df\M{x}{\Pi_{I,B}}$ using \eqref{eqn:splinerec-2}. For every $b\in B$ and every point $x\in\mathbb{R}^d$, if $\overline{M}_b(x)\neq 0$, then there is exactly one edge $\Pi_{I^\prime,B^\prime}\rightarrow\Pi_{I,B}$ with $B\setminus\{b\}=B^\prime\setminus\{b^\prime\}$ and either $I=I^\prime$, $I=I^\prime\sqcup\{b^\prime\}$,  $I^\prime=I\sqcup\{b\}$ or $I\sqcup\{b\}=I^\prime\sqcup\{b^\prime\}$. Suppose that the values of $\M{x}{(a_i)_{i\in I^\prime\sqcup B^\prime}}$ and $\M{x}{(a_i)_{i\in I^\prime\sqcup B^\prime\setminus\{b^\prime\}}}$ for all $b^\prime\in B^\prime$ are known. Are we able to compute the value of $\overline{M}_b$? The answer depends on which case is realized. In particular:
\begin{enumerate}[label=(\roman*)]
    \item If $I=I^\prime$, then $\overline{M}_b=\M{x}{(a_i)_{i\in I^\prime\sqcup B^\prime\setminus\{b^\prime\}}}$, which is known;
    \item if $I=I^\prime\sqcup\{b^\prime\}$, then $\overline{M}_b=\M{x}{\Pi_{I^\prime,B^\prime}}$, which is also known;
    \item if $I\sqcup\{b\}=I^\prime\sqcup\{b^\prime\}$, then $\overline{M}_b$ can be computed from the set of known values $\M{x}{(a_i)_{i\in I^\prime\sqcup B^\prime\setminus\{b^\prime\}}}$, $b^\prime\in B^\prime$ via a single application of \eqref{eqn:splinepar}.
\end{enumerate}
However, in the case $I^\prime=I\sqcup\{b\}$, there seems to be no obvious way to directly obtain $\overline{M}_b$. If this happens only for a single $b\in B$, then it is still possible to obtain $\overline{M}_b$ via \eqref{eqn:splinerec-2}, after noticing that $\overline{M}=\M{x}{(a_i)_{i\in I^\prime\sqcup B^\prime\setminus\{b^\prime\}}}$. In general, however, this case can happen more than once for a given point $x$ and a given spline $\M{x}{\Pi_{I,B}}$ if $d\geq 3$, making it essentially impossible to build an efficient recurrent evaluation scheme without the use of some auxiliary functions.

We propose here a slightly different construction, based on the following observation. First, notice that the problematic case $I^\prime=I\sqcup\{b\}$ cannot arise if $\M{x}{\Pi_{I,B}}$ is a spline of maximal degree for $\mathcal{P}$ (see Figure \ref{fig:zonobase}). However, if we consider a zonotopal tiling $\mathcal{P}_{I,B}$ of the zonotope $Z(V_{I,B})$ built on the reduced point configuration $A_{I,B}\df(a_i)_{i\in I\sqcup B}$, then $\M{x}{\Pi_{I,B}}$ can indeed be obtained from any maximal-degree tile of $\mathcal{P}_{I,B}$. Thus, if in the evaluation of each spline $\M{x}{\Pi_{I,B}}$ we use the reduced tiling $\mathcal{P}_{I,B}$, the problematic case $I^\prime=I\sqcup\{b\}$ cannot occur, and neither can the case  $I^\prime\sqcup\{b^\prime\}=I$. Notice that an induced tiling $\mathcal{P}_{I,B}$ of $Z(V_{I,B})$ can simply be obtained from $\mathcal{P}$ via Lemma \ref{lem:inducedtil}.

The reasoning of the previous paragraph suggests a simple procedure to build a set of auxiliary spline functions that are sufficient to compute, via recurrence, the value of any function $\M{x}{\Pi_{I,B}}$:
\begin{enumerate}[label=(\roman*)]
    \item Build the tiling $\mathcal{P}_{I,B}$ induced by $\mathcal{P}$ on the reduced point configuration $A_{I,B}\df(a_i)_{i\in I\sqcup B}$ via Lemma \ref{lem:inducedtil};\label{p:start}
    \item For each $b\in B$, find the unique tile $\Pi_{I^\prime,B^\prime}\in\mathcal{P}_{I,B}$, if any, such that $B\cap B^\prime=B\setminus\{b\}$. If the tile exists, the value of $\M{x}{(a_i)_{i\in I\sqcup B\setminus\{b\}}}$ can then be computed from the values of $\M{x}{\Pi_{I^\prime,B^\prime}}$ and $\M{x}{(a_i)_{i\in I^\prime\sqcup B^\prime\setminus\{b^\prime\}}}$, $b^\prime\in B^\prime$, either directly or through \eqref{eqn:splinepar}, otherwise the value is zero;\label{p:neig}
    \item Store the subsets $(I^\prime,B^\prime)$ found in step \ref{p:neig}, and repeat the same process from step \ref{p:start} starting from each corresponding tile $\Pi_{I^\prime,B^\prime}$. 
\end{enumerate}

The set of stored subsets $(I^\prime,B^\prime)$ obtained during this process corresponds to a set of auxiliary spline functions that are sufficient to compute the value of the spline $\M{x}{\Pi_{I,B}}$ for all $x$. Applying this process to all tiles $\Pi_{I,B}\in\mathcal{P}^{(k)}$ then yields a complete set of auxiliary functions sufficient for the evaluation of all the basis functions of order $k$ via \eqref{eqn:splinerec-2} and \eqref{eqn:splinepar}. Notice that the same couple $(I^\prime,B^\prime)$ can be obtained starting from multiple basis functions, in which case, it should obviously be stored only once.

So far, we have not detailed how the subsets corresponding to the tiles connected to $\Pi_{I,B}$ in the induced tiling $\mathcal{P}_{I,B}$ can be found efficiently in step \ref{p:neig}. Naively, one can start from the knowledge of the whole tiling $\mathcal{P}$ and apply Lemma \ref{lem:inducedtil}, but this is obviously computationally infeasible in most applications. Thankfully, in the case of regular tilings, there is a more efficient way to compute them.
\begin{lemma}
\label{lem:delwigo}
Let $\mathcal{P}(h)$ be a regular fine zonotopal tiling of $Z(V)$ with height function $h$, and let $\Pi_{I,B}$ and $\Pi_{I^\prime,B^\prime}$ be two of its tiles, sharing a facet $\Pi_{J,C}$ with normal vector $N_C$. Define for convenience:
\begin{empheq}{alignat=1}
\begin{aligned}
    \sigma_{ij}&\df\sign\left(\det((a_c,h(a_c),1)_{c\in C},(a_i,h(a_i),1),(a_j,h(a_j),1))\right),\\
    \sigma_i&\df\sign\left(\det((a_c,1)_{c\in C},(a_i,1))\right).
\end{aligned}
\end{empheq}
Then $b^\prime\in I$ if and only if $\sigma_{bb^\prime}\cdot \sigma_b>0$, $b\in I^\prime$ if and only if $\sigma_{bb^\prime}\cdot \sigma_{b^\prime}<0$, and, choosing the orientation of $N_C$ such that $\scal{N_C}{(x,1)}=\det((a_c,1)_{c\in C},(x,1))$, $\sign(\scal{N_C}{z-z^\prime})=\sigma_{bb^\prime}\cdot \sigma_b\cdot \sigma_{b^\prime}$ for all $z\in\Pi_{I,B}$, $z^\prime\in\Pi_{I^\prime,B^\prime}$.
\end{lemma}
\begin{proof}
The first two facts follow immediately from the Delaunay property \eqref{eqn:delcond}, since, if $\sigma_b>0$, then $b^\prime\in I$ if and only if $\sigma_{bb^\prime}>0$, and the same is true if both signs are reversed. The same reasoning applies to the condition $b\in I^\prime$ using $\sigma_{b^\prime b}=-\sigma_{bb^\prime}$ and $\sigma_{b^\prime}$. If we now consider the representative points $z\in\Pi_{I,B}$ and $z^\prime\in\Pi_{I^\prime,B^\prime}$ defined as in \eqref{eqn:tilebary}, we can express their difference as
\begin{equation}
    z-z^\prime=\frac{1}{2}\left(\sigma_{bb^\prime}\sigma_{b^\prime}v_b+\sigma_{bb^\prime}\sigma_b v_{b^\prime}\right),
\end{equation}
and therefore 
\begin{equation}
\label{eqn:signinterm}
    \sign(\scal{N_C}{z-z^\prime})=\frac{1}{2}\sign\left(\sigma_{bb^\prime}\sigma_{b^\prime}\scal{N_C}{(a_b,1)}+\sigma_{bb^\prime}\sigma_b\scal{N_C}{(a_{b^\prime},1)}\right)
\end{equation}
but since $b$ and $b^\prime$ are on the same side of $\aff(\{a_c\}_{c\in C})$ if and only if $0<\sigma_b\cdot\sigma_{b^\prime}=(\sigma_{bb^\prime}\sigma_b)\cdot(\sigma_{bb^\prime}\sigma_{b^\prime})$, the two terms in the sum  on the right hand side of \eqref{eqn:signinterm} always have the same sign, and we can thus rewrite \eqref{eqn:signinterm} as
\begin{equation}
    \frac{1}{2}\left(\sigma_{bb^\prime}\sigma_{b^\prime}\sign(\scal{N_C}{(a_b,1)})+\sigma_{bb^\prime}\sigma_b\sign(\scal{N_C}{(a_{b^\prime},1)})\right)=\sigma_{bb^\prime}\cdot\sigma_b\cdot\sigma_{b^\prime},
\end{equation}
since $\sign(\scal{N_C}{(a_b,1)})=\sigma_b$, and similarly for $b^\prime$. This completes the proof.
\end{proof}

In the case of regular tilings, Lemma \ref{lem:delwigo} can be used to build any induced tiling $\mathcal{P}_{I,B}$, its adjacency graph and the induced orientations simply by taking the collection $\mathcal{B}\df\{B^\prime\subseteq I\sqcup B:\norm{B^\prime}=d+1,\detp{B^\prime}\neq 0\}$
of all affinely independent subsets of size $d+1$ of $(a_i)_{i\in I\sqcup B}$, and using for each subset $B^\prime$ the signs $\sigma_{bb^\prime}$, $\sigma_b$ and $\sigma_{b^\prime}$, $b^\prime\in B^\prime$ to construct the associated subset $I^\prime$ and form the tile $\Pi_{I^\prime,B^\prime}\in\mathcal{P}_{I,B}$. The evaluation graph for $\Pi_{I,B}$ will then contain all the tiles directly adjacent to $\Pi_{I,B}$ in $\mathcal{P}_{I,B}$. Notice that, when all auxiliary functions are taken into account, the splines of degree zero do {\em not} constitute in general a triangulation of $\ch(A)$. However, it is still possible to build search trees capable of efficiently finding all the (possibly overlapping) simplices that contain a given point $x$, for example using structures such as bounding volumes hierarchies (BVH), of which the $R$-tree and $R^\star$-tree \cite{guttman1984r,beckmann1990r} are prominent examples. We illustrate the construction of auxiliary functions and the corresponding evaluation obtained via the process outlined above in Figs~\ref{fig:splineaux} and \ref{fig:splineeval} respectively.

\begin{figure}[htbp]
    \centering
    \subfloat{\includegraphics[width=0.40\linewidth]{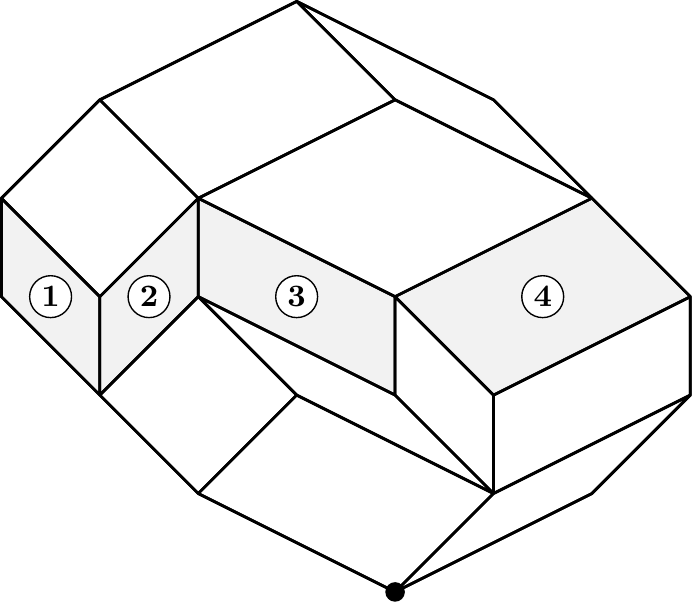}}\hspace{0.05\textwidth}%
    \subfloat{\includegraphics[width=0.40\linewidth]{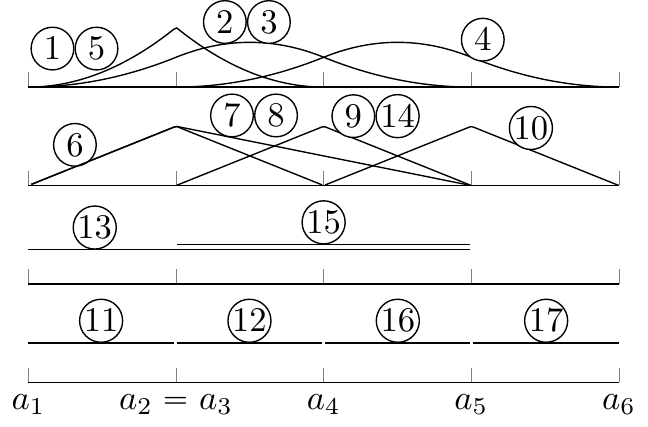}}%
   \newline
    \centering
    \subfloat{\includegraphics[width=0.20\linewidth]{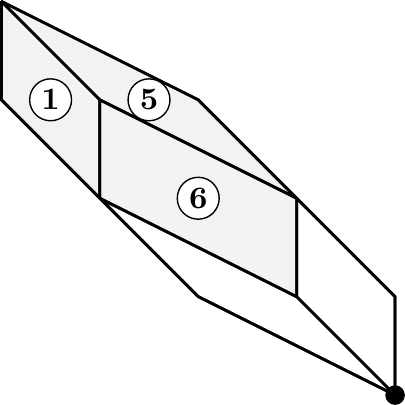}}\hspace{0.05\textwidth}%
    \subfloat{\includegraphics[width=0.20\linewidth]{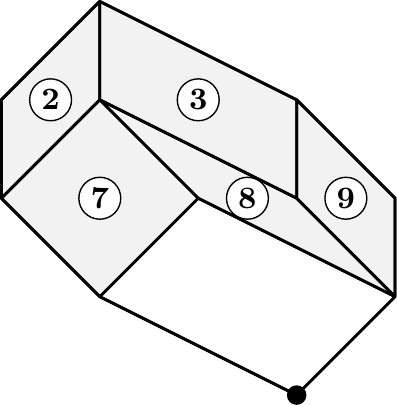}}\hspace{0.05\textwidth}%
    \subfloat{\includegraphics[width=0.20\linewidth]{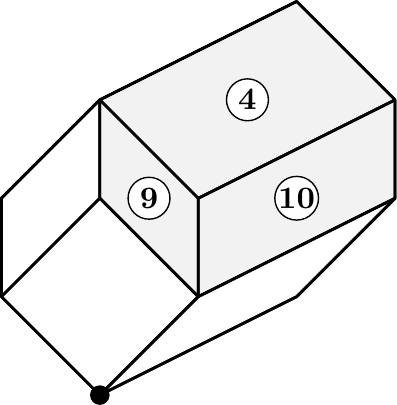}}\hspace{0.05\textwidth}%
\newline
    \centering
    \subfloat{\includegraphics[width=0.15\linewidth]{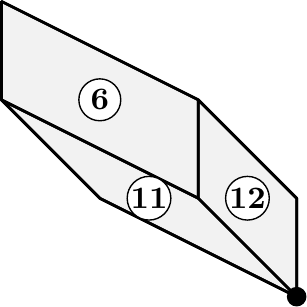}}\hspace{0.05\textwidth}%
    \subfloat{\includegraphics[width=0.20\linewidth]{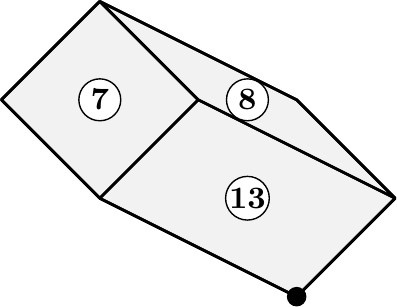}}\hspace{0.05\textwidth}%
    \subfloat{\includegraphics[width=0.10\linewidth]{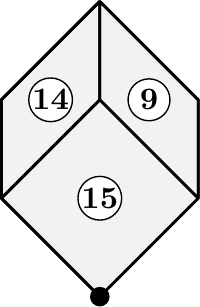}}\hspace{0.05\textwidth}%
    \subfloat{\includegraphics[width=0.15\linewidth]{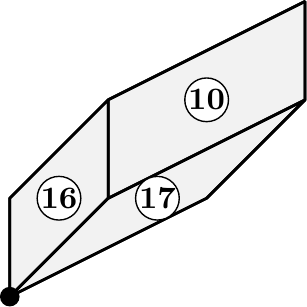}}%
    \caption{{\em Top left}: A regular fine zonotopal tiling and the associated spline space over the point configuration of Figure \ref{fig:zonobase}, with the tiles corresponding to the splines of degree $k=2$ (i.e., $\mathcal{P}^{(2)}$) highlighted and numbered from $1$ to $4$. {\em Top right}: corresponding spline functions and auxiliary functions, numbered $5$ through $17$, computed by the process of Section \ref{sec:del:eval}. {\em Bottom}: the induced zonotopal tilings $\mathcal{P}_{I,B}$ encountered during the construction of auxiliary spline functions. Highlighted tiles correspond to stored functions.}
    \label{fig:splineaux}
\end{figure}

\begin{figure}[htbp]
    \centering
   \hspace*{\fill}%
    \subfloat{\includegraphics[width=0.40\linewidth]{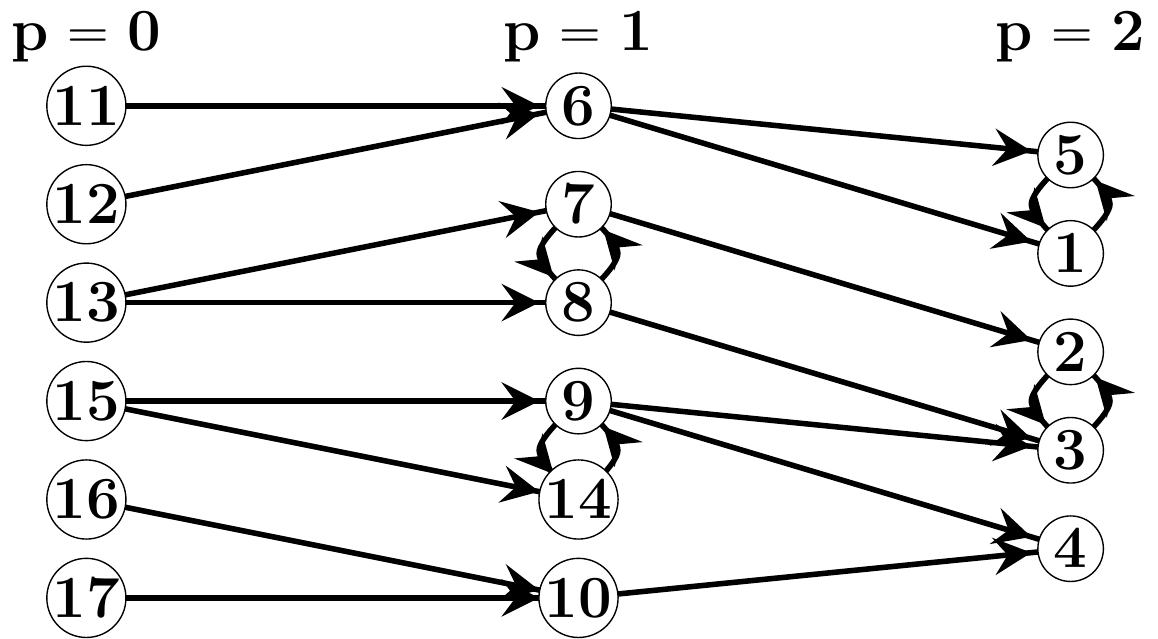}}\hspace*{\fill}%
    \newline
    \subfloat{\includegraphics[width=0.40\linewidth]{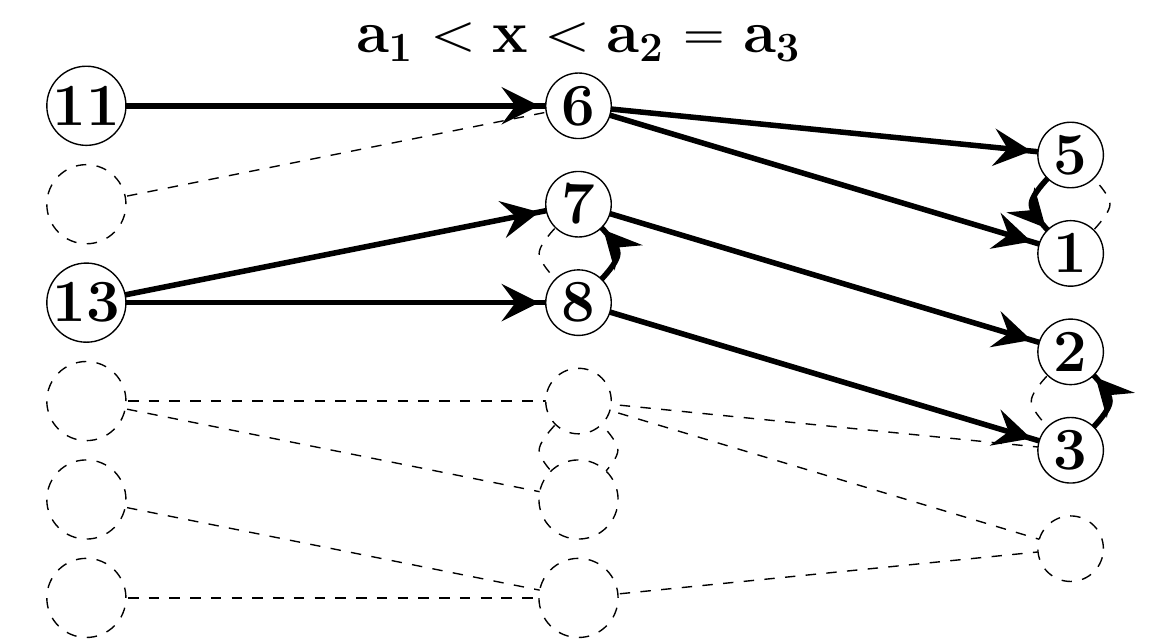}}\hspace{0.1\linewidth}%
    \subfloat{\includegraphics[width=0.40\linewidth]{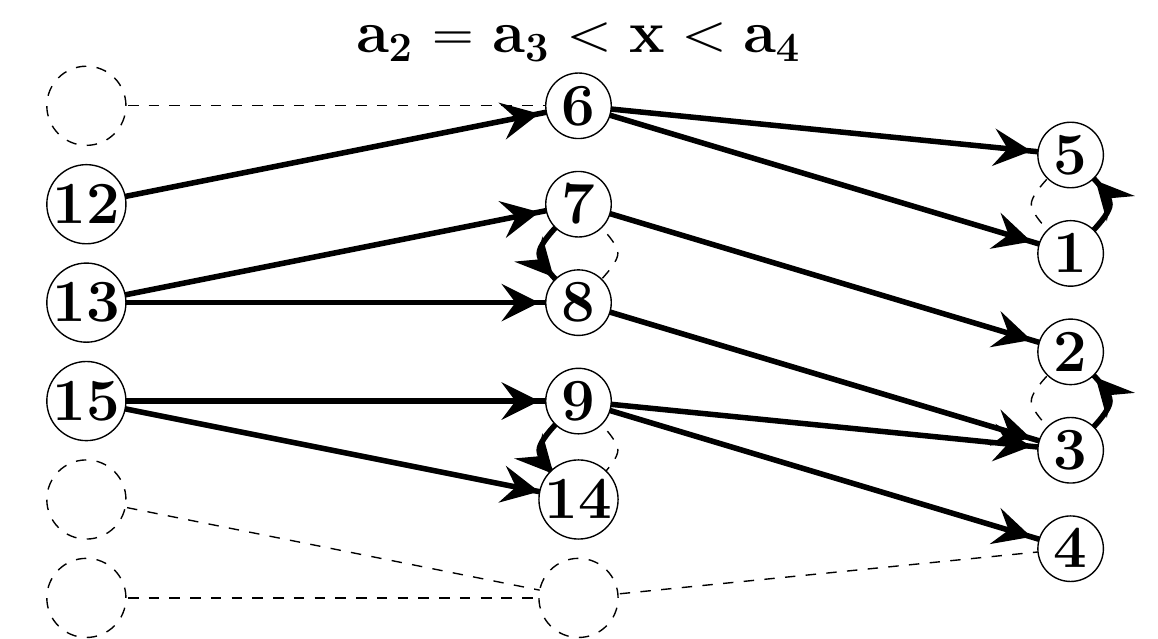}}\hfill%
    \subfloat{\includegraphics[width=0.40\linewidth]{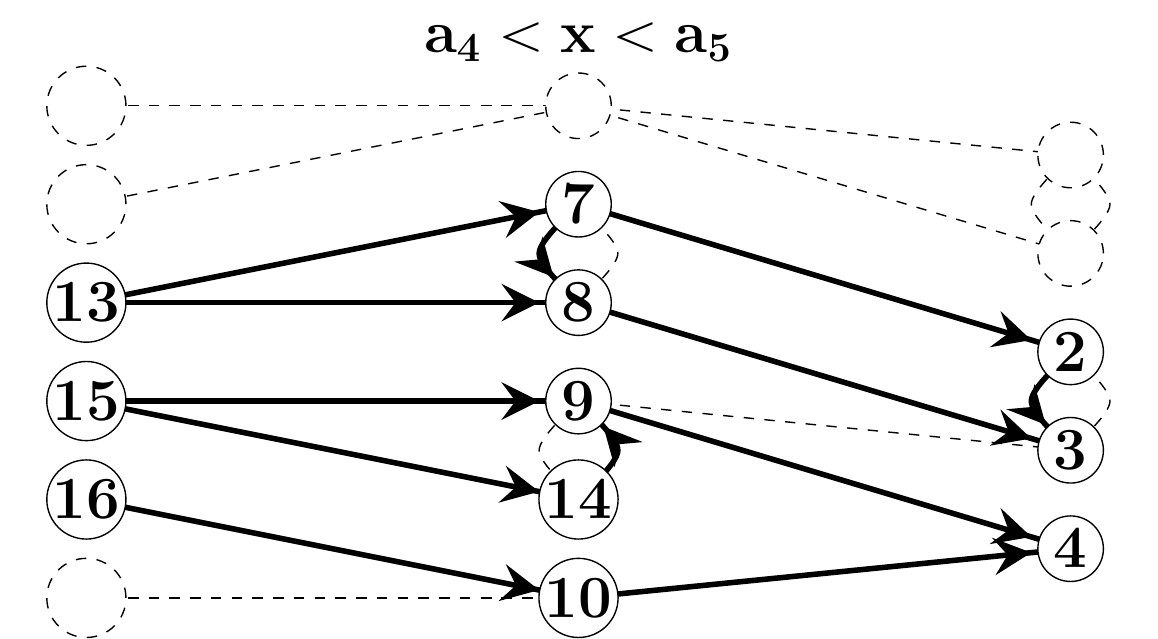}}\hspace{0.1\linewidth}%
    \subfloat{\includegraphics[width=0.40\linewidth]{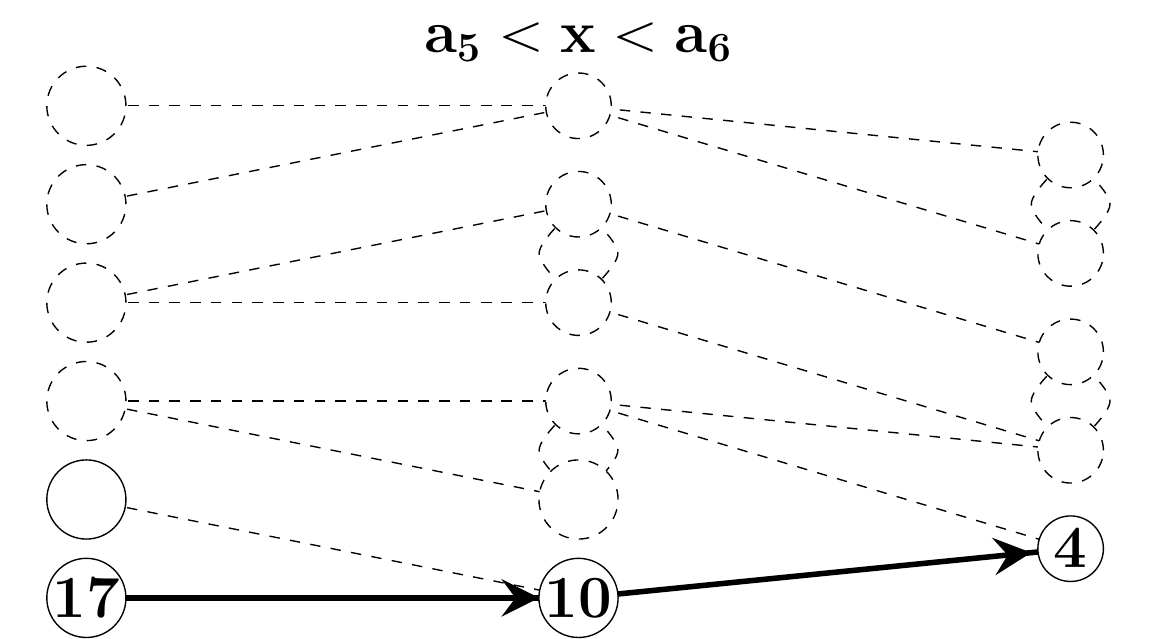}}\hfill%
    \caption{{\em Top}: the complete graph containing all the auxiliary functions obtained via the construction presented in Section \ref{sec:del:eval} in the case of the example of Figure \ref{fig:splineaux}. {\em Bottom}: The actual evaluation graph obtained when computing the value of the spline functions at different locations $x$.}
    \label{fig:splineeval}
\end{figure}

We end this section with a couple of final considerations. First, notice that it is not necessary to explicitly prove that the evaluation graph is acyclic, as this is evident from its construction. In particular, the evaluation graph for splines of order $k$ clearly generates a $k$-partite oriented graph, to which some connections between splines of the same order are added (Figure \ref{fig:splineeval}). Since the connections among this subset of tiles are the same as those in the full adjacency graph $\mathcal{G}$ of $\mathcal{P}$, no cycle can be created by the orientation $o_x$ induced by any point $x$.

Second, notice that in the special case where {\em every} point in $A$ is repeated at least $k+1$ times, the construction process of Theorems \ref{thm:constructiongen} and \ref{thm:constructiondel} yields the usual Bernstein-Bézier functions \cite{lyche2000p} over a triangulation of $\ch(A)$, and the evaluation graph reduces to the usual de Casteljau algorithm \cite{de1963courbes} over each simplex.

Finally, notice that, as can be gleaned from Figure \ref{fig:splineaux}, the procedure outlined here does not lead in general to a minimal amount of auxiliary spline functions. In particular, each tile $\Pi_{I,B}$ for which there is an index $i\in I$ such that $a_i\not\in\ch(\{a_b\}_{b\in B})$ can lead to an increased number of auxiliary functions. How often this happens is determined by the chosen height function $h$, either globally or locally in each induced tiling $\mathcal{P}_{I,B}$, and is related to the presence of {\em slivers}, i.e., simplices with skewed aspect ratios, in the associated weighted Delaunay triangulations. Some techniques exist to optimize the Delaunay height function in order to reduce the number of these elements, see e.g. \cite{edelsbrunner2002experimental,mullen2011hot}. We defer to a future work the investigation of how these techniques can help optimize the number of auxiliary functions required in the evaluation of simplex splines.

\section{Conclusions}
\label{sec:concl}
We have uncovered an interesting combinatorial structure capable of producing spaces of polynomial-reproducing multivariate (simplex) splines built atop any point configuration $A$, which ties them to the well studied fine zonotopal tilings of the associated zonotope $Z(V)$. This correspondence allows to generalize the set of known multivariate spline spaces and to adapt a known construction algorithm to a more general setting. When the tiling is regular, its adjacency graph provides a way to efficiently determine all the spline functions supported on any given point $x$, and to devise a recurrence evaluation scheme that reuses some intermediate results, thus providing a useful first step in the practical application of simplex spline bases in approximation and analysis.

Only fine zonotopal tilings have been explored in the present work. Possible connections between more general zonotopal tilings and other kinds of multivariate splines, such as box splines or more general polyhedral splines \cite{goodman1990polyhedral,de2013box2} might be possible by generalizing this restriction.

From a computational standpoint, it is possible that the correspondence uncovered in the present work can be used to obtain further optimized algorithms for multivariate splines. Two aspects in particular deserve a particular attention in our opinion. 

First, the evaluation scheme proposed in this work does not guarantee a minimal number of auxiliary functions. On the other hand, optimized weighted Delaunay triangulations coming from computer graphics applications (see e.g. \cite{edelsbrunner2002experimental,mullen2011hot}) could provide more suitable height functions, significantly improving the efficiency of the evaluation algorithm.

Second, the freedom given by the possibility of constructing spline bases over point sets with repeated knots can be exploited to build bases of splines with variable regularity and localized or arbitrarily-shaped discontinuities, with interesting applications in function approximation and numerical analysis.
%printbibliography

\addtocontents{toc}{\protect\setcounter{tocdepth}{0}}
\section*{Acknowledgements}
This work is supported by the Inria - Total S.E. strategic action ``Depth Imaging Partnership" (\url{http://dip.inria.fr}).
\addtocontents{toc}{\protect\setcounter{tocdepth}{1}}
\printbibliography

\end{document}